\documentclass{amsart}

\usepackage{amsmath,amsthm,amsfonts,amssymb,amscd,verbatim,latexsym}
\usepackage{hyperref}
\usepackage{xcolor}
\hypersetup{
    colorlinks,
    linkcolor={blue!80!black},
    citecolor={blue!80!black},
}
\usepackage{enumerate}
\usepackage{dsfont}
\usepackage{array}
\usepackage{color}
\usepackage{bbold}
\usepackage{bbm}

\newcommand{\cA}{\mathcal{A}}
\newcommand{\cB}{\mathcal{B}}
\newcommand{\bG}{\mathbf{G}}

\DeclareMathOperator{\SL}{SL}
\DeclareMathOperator{\Ext}{Ext}
\newcommand{\kk}{\Bbbk}

\newcommand{\id}{\operatorname{id}}
\newcommand{\Hom}{\operatorname{Hom}}

\newcommand{\ZZ}{\mathbb{Z}}

\newcommand{\NN}{\mathbb{N}}

\newtheorem{theorem}{Theorem}[section]
\newtheorem{lemma}[theorem]{Lemma}

\newtheorem{proposition}[theorem]{Proposition}

\newtheorem{question}[theorem]{Question}
\theoremstyle{definition}
\newtheorem{definition}[theorem]{Definition}

\newtheorem{notation}[theorem]{Notation}
\renewcommand{\i}{\mathbbm{i}}
\newtheorem{conjecture}[theorem]{Conjecture}

\newtheorem{remark}[theorem]{Remark}

\numberwithin{equation}{section}

\begin{document}

\title[Three infinite families of reflection Hopf algebras]{Three infinite families of reflection Hopf algebras}


%    author one information

\author{Luigi Ferraro}
\address{Wake Forest University, Department of Mathematics and Statistics, P. O. Box 7388, Winston-Salem, North Carolina 27109}
\email{ferrarl@wfu.edu}

\author{Ellen Kirkman}
\address{Wake Forest University, Department of Mathematics and Statistics, P. O. Box 7388, Winston-Salem, North Carolina 27109} 
\email{kirkman@wfu.edu}

\author{W. Frank Moore}
\address{Wake Forest University, Department of Mathematics and Statistics, P. O. Box 7388, Winston-Salem, North Carolina 27109}
\email{moorewf@wfu.edu}

\author{Robert Won}
\address{University of Washington, Department of Mathematics, Box 354350, Seattle, Washington 98195}
\email{robwon@uw.edu}

%    \subjclass is required.

\date{}

\dedicatory{}

%    Abstract is required.
\begin{abstract}
Let $H$ be a semisimple Hopf algebra acting on an Artin-Schelter regular algebra $A$, homogeneously,  inner-faithfully, preserving the grading on $A$, and so that $A$ is an $H$-module algebra. When the fixed subring $A^H$ is also  AS regular, thus providing a generalization of the Chevalley-Shephard-Todd Theorem, we say that $H$ is a reflection Hopf algebra for $A$.  We show that
each of the semisimple Hopf algebras $H_{2n^2}$ of Pansera, and $\mathcal{A}_{4m}$ and $\mathcal{B}_{4m}$ of Masuoka is a reflection Hopf algebra for an AS regular algebra of dimension 2 or 3.
\end{abstract}

\maketitle

\section*{Introduction}
Throughout let $\kk = \mathbb{C}$, and denote the square root of $-1$ by $\i$.  A 
finite subgroup $G$ of GL$_n(\kk)$, acting linearly as graded automorphisms on a 
(commutative) polynomial ring $A = \kk[x_1, \dots, x_n]$, is called a {\em reflection 
group} if $G$ is generated by  elements $g \in G$, which act on the vector space 
$\bigoplus \kk x_i$ with fixed subspace of codimension 1; this condition is equivalent 
to the condition that all the eigenvalues of  $g$ are $1$, with the exception of one 
eigenvalue that is a root of unity (sometimes such elements $g$ are called {\em 
pseudoreflections} when the exceptional eigenvalue is not $-1$).  Chevalley 
\cite{Chev} and Shephard and Todd \cite{SheTod} showed that over a field of 
characteristic zero, a group $G$  is a reflection group if and only if the invariant 
subalgebra $A^G$ is a polynomial ring, and Shephard and Todd  \cite{SheTod} presented 
a complete classification of the reflection groups into three infinite families (the 
cyclic groups, the symmetric groups, and the groups $G(m, p, n)$ for positive 
integers $m$, $p$, and $n$, where $p$ divides $m$), and thirty-four exceptional groups.  Reflection 
groups have played an important role in many contexts, including in representation 
theory, combinatorics, commutative ring theory, and algebraic geometry. 

There has been interest in extending the Chevalley-Shephard-Todd Theorem to a 
noncommutative context (replacing the commutative polynomial ring with a 
noncommutative algebra $A$), and in \cite[Definition 2.2]{KKZ1} an analog of a 
reflection (called a {\em quasi-reflection} in that paper) was defined for a graded 
automorphism $g$ of an Artin-Schelter regular (AS regular) algebra $A$ that is 
generated in degree 1 (Definition \ref{def:asregular}).  When such an AS regular 
algebra $A$  is commutative, it is isomorphic to a commutative polynomial ring, so 
this particular noncommutative setting generalizes the classical commutative 
polynomial algebra case.  Moreover, examples suggest that the proper analog of a 
reflection group for $A$ is a group $G$ such that the invariant subalgebra $A^G$ is 
also AS regular. The extended notion of the definition of ``reflection"  of \cite{KKZ1}  
(which involves ``trace functions'' rather than eigenvalues)  was used in  \cite{KKZ2} 
to prove a version of the Chevalley-Shephard-Todd Theorem for groups acting on skew 
polynomial rings (and a second proof was given in \cite{BB}). Among the reflection groups for the skew polynomial ring $A = \kk_{-1}[u,v]$ are the dicyclic groups (also 
known as binary dihedral groups) $Q_{4\ell}$ generated by $a$ and $b$ with relations:
$a^{2\ell} = 1, b^{-1}ab = a^{-1}, b^2 = a^\ell$; so, for example, the quaternion 
group of order 8 is a reflection group for $A = \kk_{-1}[u,v]$. These groups are not 
among the classical reflection groups.

To extend classical invariant theory further, the group 
$G$ can be replaced by a semisimple Hopf algebra $H$ (see  \cite{KKZ3}) that acts on a  
noncommutative AS regular algebra $A$, and several extensions of results for the 
action of a finite subgroup of $\SL_{2}(\kk)$ on $\kk[u,v]$ have been proved in this context 
(e.g., \cite{CKWZ, CKWZ1, CKWZ2}). However, it has appeared 
more difficult to extend the Chevalley-Shephard-Todd Theorem to Hopf actions.  To 
this end we consider pairs $(A,H)$, where $A$ is an AS regular algebra and $H$ is a 
(finite-dimensional) semisimple Hopf algebra, equipped with an action of $H$ on $A$ 
that preserves the grading, and is inner-faithful on $A$ (meaning that 
no non-zero Hopf ideal of $H$ annihilates $A$, see Section \ref{sec:IF}), with $A$ an 
$H$-module algebra (so that the coproduct $\Delta$ of $H$ is used to compute the 
actions of elements of $H$ on products of elements of $A$).  We call $H$ a {\em 
reflection Hopf algebra for $A$} (\cite[Definition 3.2]{KKZ6}) 
if the ring of invariants $A^H$ is AS regular.  

In \cite[Examples 7.4 and 7.6]{KKZ7} it was shown that the Kac-Palyutkin algebra is a 
reflection Hopf algebra for both $A = \kk_{-1}[u,v]$ and $A = \kk_{\i}[u,v]$.  In 
\cite{KKZ6} the case of a Hopf algebra of the form $H = \kk G^\circ$, the dual of a 
group algebra (or equivalently, a group coaction) was considered, and some {\em dual 
reflection groups} were constructed. In \cite{FKMW} the sixteen non-trivial Hopf 
algebras of dimension sixteen classified by Kashina \cite{kashina} were considered, 
and the methods used in this paper were used to determine which are reflection Hopf  
algebras for AS regular algebras of dimension 2 and 3.  
 
In this paper we consider three infinite families of Hopf algebras: the Hopf algebras $H_{2n^2}$ of dimension $2n^2$ defined by Pansera \cite{P}, and the two families $\mathcal{A}_{4m}$ and $\mathcal{B}_{4m}$ of Hopf algebras of dimension $4m$ defined by Masuoka \cite{M}.  We begin in section \ref{sec:KP} by considering the Kac-Palyutkin algebra, which occurs as $H_8$ ($n=2$) in the Pansera construction, as well as $\mathcal{B}_8$ ($m=2$) in one of the Masuoka constructions.  The Pansera construction is a generalization of the Kac-Palyutkin Hopf algebra
and is an extension of the form:
$$ \kk \rightarrow \kk [\mathbb{Z}_n \times \mathbb{Z}_n] \rightarrow H \rightarrow \kk \mathbb{Z}_2 \rightarrow  \kk.$$  The Hopf algebras $\mathcal{A}_{4m}$ and $\mathcal{B}_{4m}$ can be viewed as deformations of $\kk Q_{4m}$ (see \cite{BichonNatale}), and are extensions of the form:
$$\kk \rightarrow \kk^{\mathbb{Z}_2} \rightarrow H \rightarrow \kk D_{2n} \rightarrow \kk.$$
The examples of reflection Hopf algebras that we have computed indicate that there are an abundance of examples.  The properties that characterize such a pair $(A,H)$ are not clear, and invite further investigation.  One obvious question is:
\begin{question}
When is a bicrossed product $H= K \#_\sigma^\tau \overline{H}$ a reflection Hopf algebra for some AS regular algebra $A$?
\end{question}
The method that is used in this paper is as follows. 
First, we compute the Grothendieck ring of finite-dimensional $H$-modules for each Hopf algebra $H$. The results are summarized in the following table.
\begin{center}
~\\
 \begin{tabular}{|c|c|}
\hline
Hopf Algebra $H$ & $K_0(H)$\\
\hline\hline
$H_{2n^2}$ & Theorem \ref{thm:k0h2n2} \\
\hline
$\mathcal{B}_{4m}$ & Theorem \ref{K0B4m} \\
\hline
$\mathcal{A}_{4m}$  where $m$ is odd & Theorem \ref{k04modd} \\
\hline
$\mathcal{A}_{4m}$  where $m$ is even & Theorem \ref{K0A4mEven}\\
\hline
\end{tabular}
\end{center}
\vspace*{.1in}
Using the fusion relations in the Grothendieck ring of $H$, we construct AS regular algebras $A$ on which $H$ acts inner-faithfully.  In the cases of these three infinite families there are always such AS regular algebras of dimension two or three. The table below lists each of the theorems where the inner-faithful representations of $H$ are presented.
\begin{center}
\vspace*{.1in}
 \begin{tabular}{|c|c|c|}
\hline
Hopf Algebra $H$ & Inner-Faithful Reps& Dimension \\
\hline\hline
$H_{2n^2}$ & Theorem \ref{thm:H2n2IF} & $2$ \\
\hline
$\mathcal{B}_{4m}$ & Theorem \ref{thm:B4mIF} & $2$ \\
\hline
$\mathcal{A}_{4m}$  where $m$ is odd & Theorem \ref{A4mOddIF} & $2$ \\
\hline
$\mathcal{A}_{4m}$  where $m$ is even & Theorem \ref{A4mEvenIF} & $3$ \\
\hline
\end{tabular}
\end{center}
\vspace*{.1in}
Using the smallest dimension AS regular quadratic algebras $A$ on which $H$ acts inner-faithfully, we compute the fixed ring $A^H$ and determine when it is also AS regular.  We obtain the following theorem:
\begin{theorem}
The following Hopf algebras are reflection Hopf algebras for the given AS regular algebras.
\begin{enumerate}
\item (Theorem \ref{thm:H2n2refhopf}) $H_{2n^2}$ acting by $\pi_{i,j}(u,v)$ (Theorem \ref{RepH2n2}) on the two-dimensional AS regular algebra
$$A^-=\frac{\kk\langle u,v\rangle}{(p^{i^2-j^2}uv-vu)},\quad\mathrm{where}\;(i^2- j^2,n)=1{\text{ and } 0 \leq i < j \leq n-1},$$ 
for $p = -e^{{\pi \i}/{n}} = e^{{(n+1) \pi \i}/{n}}$.
\item (Theorem \ref{thm:H2n2fixed+}) $H_{2n^2}$ acting by $\pi_{i,j}(u,v)$ (Theorem \ref{RepH2n2}) on the two-dimensional AS regular algebra
$$A^+=\frac{\kk\langle u,v\rangle}{(p^{i^2-j^2}uv+ vu)},\quad\mathrm{where}\;(i^2- j^2, n) = 1 {\text{ and } 0 \leq i < j \leq n-1},$$ 
for $p = -e^{\pi \i/n} = e^{(n+1) \pi \i/{n}}$.
\item (Theorem \ref{B4mfixed}) $\mathcal{B}_{4m}$
acting by $\pi_{i}(u,v)$ (Proposition \ref{prop:repsB4m}) on the two-dimensional AS regular algebra
\[
A^{-}=\frac{\kk\langle u,v\rangle}{(u^2-\lambda^iv^2)},\quad \text{ where } \;(i,2m)=1 {\text{ and } i = 1, \ldots, m-1},
\]
for $\lambda = e^{{\pi \i}/{m}}$. 
\item (Theorem \ref{thm:A4mOddFR}) $\mathcal{A}_{4m}$ for $m$ odd, acting by $ \pi^{-1}_i(u,v)$ (Proposition \ref{prop:repA4m}) on the two-dimensional AS regular algebra
$$A^-=\frac{\kk\langle u,v\rangle}{(u^2-\lambda^iv^2)}, \quad \text{ where } \;(i,m)=1 {\text{ and } i = 1, \ldots, (m-1)/2},$$
for $\lambda = e^{{2\pi \i}/{m}}$.
\item (Theorem \ref{A4mfixedringsEven}) $\mathcal{A}_{4m}$ for $m$ even, acting by  $\pi^{+1}_i(u,v) \otimes T_{\varepsilon,\varepsilon,-1}(t) $ 
(Proposition \ref{prop:repA4mEven}) on the three-dimensional AS regular algebras
$$A_{1,\varepsilon}^{-}=\frac{\kk \langle u,v\rangle}{(uv- vu)} [t;\sigma],\quad\sigma=\begin{pmatrix} 0&1\\\lambda^i&0\end{pmatrix},\quad\mathrm{where}\;(i, m) = 1$$ {  and $ i = 1, \ldots, m/2-1$,}
for $\lambda = e^{{2\pi \i}/{m}}$ and $\varepsilon = \pm 1$. 
\item (Theorem \ref{A4mfixedringsEven}) $\mathcal{A}_{4m}$ for {  $m\equiv 0\pmod{4}$}, acting by $\pi^{-1}_i(u,v) \otimes T_{\varepsilon,\varepsilon,-1}(t) $  (Proposition \ref{prop:repA4mEven}) on the three-dimensional AS regular algebras
$$A_{2,\varepsilon}^{-}=\frac{\kk \langle u,v\rangle}{(u^2 - \lambda^iv^2)} [t;\sigma],\quad\sigma=\begin{pmatrix} 0&1\\\lambda^i&0\end{pmatrix},\quad\mathrm{where}\;(i, m) = 1$$ { and  $i = 1, \ldots, m/2-1,$}
for $\lambda = e^{2\pi \i/m}$ and $\varepsilon = \pm 1$. 
\item (Theorem \ref{A4mfixedringsEven}) $\mathcal{A}_{4m}$ for { $m\equiv 2\pmod{4}$}, acting by $\pi^{-1}_i(u,v) \otimes T_{\varepsilon,-\varepsilon,1}(t) $ (Proposition \ref{prop:repA4mEven}) on the three-dimensional AS regular algebras
$$A_{5,\varepsilon}^{-}=\frac{\kk \langle u,v\rangle}{(u^2 - \lambda^iv^2)} [t;\sigma],\quad\sigma={\begin{pmatrix} 1&0\\0 & -1 \end{pmatrix}},\quad\mathrm{where},\; (i, m) = 1$$ { and  $i = 1, \ldots, m/2-1$},
for  $\lambda = e^{2\pi \i/m}$ and $\varepsilon = \pm 1$. 
\end{enumerate}
\end{theorem}
For each of fixed rings  the product of the degrees of the minimal generators of the invariants is equal to the dimension of the Hopf algebra.  The following conjecture is true for actions of  reflection groups on a commutative polynomial ring, and in all the group and Hopf action examples we have computed:

\begin{conjecture}
Let $A$ be an AS regular 
algebra, and $H$ a semisimple reflection Hopf algebra for $A$.
Then the product of the degrees of any 
homogeneous minimal generating set of the algebra $A^H$ is $\dim H$.
\end{conjecture}
The paper is organized as follows.  Background material is presented in Section \ref{background},
the Kac-Palyutkin algebra is discussed in Section \ref{sec:KP},
the Pansera algebras $H_{2n^2}$ are discussed in Section \ref{sec:H2n2}, the algebras  $\mathcal{B}_{4m}$ of Masuoka are discussed in Section \ref{sec:B4m}, and the algebras
$\mathcal{A}_{4m}$ of Masuoka are discussed in  Section \ref{sec:A4modd} ($m$ odd) and Section \ref{sec:A4meven} ($m$ even). In Section \ref{sec:ext} we note that an action of $H$ on $A$ can be extended to an Ore extension $A[t;\sigma]$ with $A[t;\sigma]^H = A^H[t;\sigma]$, so that the algebras $A$ on which a Hopf algebra acts as a reflection Hopf algebra can have arbitrarily large dimension.

%%%%%%%%%%%%%%%%%%%%%%%%%%
%  Background
%%%%%%%%%%%%%%%%%%%%%%%%%%%%%%%%%%%

\section{Background} \label{background}

We follow the standard notation for Hopf algebras, and refer to \cite{Mbook} for any undefined terminology concerning Hopf algebras. For a Hopf algebra $H$, the set of grouplike elements of $H$ is denoted $\bG(H)$.  

%%%%%%%%%%%%%%%%%%%%%%%%%%
% Subsection 1.1 AS reegular algebras
%%%%%%%%%%%%%%%%%%%%%%

\subsection{AS regular algebras}

We consider Hopf algebras that act on AS regular algebras, which are algebras possessing homological properties of commutative polynomial rings.
\begin{definition}
\label{def:asregular}
Let $A$ be a connected graded algebra. Then $A$ is
{\it Artin-Schelter (AS) regular of dimension $d$} if it satisfies the
conditions below:
\begin{enumerate}
\item[(1)] $A$ has finite global dimension $d$;
\item[(2)] $A$ has finite Gelfand-Kirillov dimension;
\item[(3)] $A$ satisfies the {\it  Gorenstein condition}, i.e.,
 $\Ext^i_A(\kk,A) = \delta_{i,d} \cdot \kk(l)$ for some
$l \in \mathbb{Z}$.
\end{enumerate}
\end{definition}

Examples of AS regular algebras include skew polynomial rings and Ore extensions of AS regular algebras; the AS regular rings of invariants we find here will either be commutative polynomial rings or Ore extensions of skew polynomial rings.  The AS regular algebras of dimensions 2 and 3 have been classified.  We will use the following well-known fact (see e.g., \cite[Lemma 1.2]{CKZ}) to show that an invariant subring is not AS regular.

\begin{lemma} \label{lem:notAS}
If $A$ is an AS regular algebra of GK dimension 2 (resp. 3), then $A$ is generated by 2 (resp. 2 or 3) elements .
\end{lemma}

We will encounter algebras of the form $\kk\langle u,v \rangle/(u^2 - cv^2)$ several times 
in this paper, so we record here the following lemma which identifies a set of monomials in 
$u$ and $v$ as a basis.
\begin{lemma} \label{lem:u2v2Basis}
Let $c$ be a nonzero element of $\kk$ and
$A = \kk\langle u,v \rangle/(u^2 - cv^2)$.  Then $A$ is AS regular, and 
the set of monomials $\{u^i(vu)^jv^\ell\}$ with $i,j$ nonnegative integers
and $\ell \in \{0,1\}$ forms a $\kk$-basis of $A$.
\end{lemma}

\begin{proof}
A straightforward computation shows that a reduced Gr\"obner basis of the ideal generated
by $u^2 - cv^2$ with respect to the graded lexicographic term order $u < v$ is given by
$v^2 - c^{-1}u^2$ and $vu^2 - u^2v$.
It follows that each of the elements in the proposed basis
are reduced with respect to this term order and are hence linearly independent.
Therefore the coefficients of the Hilbert series of $A$ are at least those of
the polynomial ring in two variables generated in degree 1.  For $d \in \kk$ such that
$d^2 = c$, the change of basis $x = u + dv$ and $y = u - dv$ shows that $A$ is a quotient
of $\kk_{-1}[x,y]$.  However, the Hilbert series calculation shows $A$ is in fact isomorphic 
to $\kk_{-1}[x,y]$ and is thus AS regular, and hence the proposed basis in fact spans $A$.
\end{proof}

%%%%%%%%%%%%%%%%%%%%%
% Subsection 1.2 Inner-faithful actions
%%%%%%%%%%%%%%%%%%%%%%%

\subsection{Inner-faithful actions} \label{sec:IF}
An $H$-module $V$  is \emph{inner-faithful} if the only
Hopf ideal that annihilates $V$ is the zero ideal.  We record the following
result which is due to Brauer \cite{Brau}, Burnside \cite{Burn} and Steinberg \cite{Stein}
in the case of a group algebra of a finite group, and due to Passman and Quinn \cite{PassQui}
in the case of a finite-dimensional semisimple Hopf algebra.  
We include a proof for the sake
of completeness. 
\begin{theorem} \label{innerfaithful}
Let $V$ be a module over a semisimple Hopf algebra $H$.
Then the following conditions are equivalent.
\begin{enumerate}
\item $V$ is an inner-faithful $H$-module,
\item The tensor algebra $T(V)$ is a faithful $H$-module,
\item Every simple $H$-module appears as a direct summand of $V^{\otimes n}$
for some $n$.
\end{enumerate}
\end{theorem}

\begin{proof}
If $V$ is inner-faithful, then \cite[Corollary 10]{PassQui} shows that (3) holds.
If (3) holds and if $IT(V) =0$, then it follows that $I$ is contained in the Jacobson
radical of $H$, which is zero.  Finally, suppose $(2)$ holds, and let $I$
be a Hopf ideal which annihilates $V$.  Then for all $v_1,\dots,v_\ell \in V$
and for all $h \in H$, one has
\begin{eqnarray*}
h . (v_1 \otimes \cdots \otimes v_\ell)
 & = & \Delta^{(\ell)}(h)(v_1\otimes\cdots\otimes v_\ell) \\
 & = & \sum h_{(1)}v_1 \otimes \cdots \otimes h_{(\ell)}v_\ell.
\end{eqnarray*}
If $h \in I$, then since $I$ is a Hopf ideal, each summand of $\Delta^{(\ell)}(h)$ contains
some tensor factor which is in $I$.  It follows that $h \in I$ annihilates $T(V)$, hence $h = 0$.
\end{proof}

This result motivates the following definition, which we use in several
of the proofs regarding inner-faithful representations that follow.
\begin{definition} \label{GenRep}
Let $X$ and $V$ be $H$-modules over a finite-dimensional semisimple Hopf algebra.
If $X$ appears as a direct summand of $V^{\otimes n}$ for some $n$, we say that
$V$ \emph{generates} $X$.
\end{definition}

If each simple representation of $H$ occurs as a direct summand of $A$ then $H$ acts faithfully on $A$.  If $H$ acts faithfully on $A$, then clearly it acts inner-faithfully. In all the examples that we have checked the following conjecture holds.
\begin{conjecture}
When a semisimple Hopf algebra $H$ acts inner-faithfully on an AS regular algebra $A$, it also acts faithfully on $A$.
\end{conjecture}

It is well-known that a two-sided ideal generated by a skew-primitive element is a Hopf ideal.  As $1-g$ is skew primitive for $g\in\bG(H)$, we have the following lemma.

\begin{lemma}\label{HopfIdeal}
Let $H$ be a Hopf algebra and $g\in\bG(H)$. Then the two-sided ideal generated by $1-g$ is a Hopf ideal.
\end{lemma}
\begin{comment}
\begin{proof}
Let $I$ be the two-sided ideal generated by $1-g$. Since $\varepsilon(1-g)=0$ then $\varepsilon(I)=0$. To prove $S(I)\subseteq I$ we notice that $1-g^{-1}=-g^{-1}(1-g)\in I$ and therefore $S(1-g)\in I$, the desired containment follows since $S$ is an anti-homomorphism. To prove that $\Delta(I)\subseteq H\otimes I+I\otimes H$ we notice that 
\[
\Delta(1-g)=\frac{1}{2}[(1-g)\otimes(1+g)+(1+g)\otimes(1-g)]\in H\otimes I+I\otimes H.
\]
If $x,y\in H$ then $\Delta(x(1-g)y)\in H\otimes I+I\otimes H$ because $I$ is a two-sided ideal.
\end{proof}
\end{comment}

%%%%%%%%%%%%%%%%%%%%%%%%%%
%Subsection 1.3 Grothendieck ring
%%%%%%%%%%%%%%%%%%%%%%%

\subsection{The Grothendieck ring}

To determine if a particular representation $V$ of a semisimple Hopf algebra $H$ is inner-faithful, by Theorem \ref{innerfaithful} it is necessary to compute the decomposition of tensor powers of $V$ into irreducible $H$-modules, i.e., to determine the fusion rules in the Grothendieck ring $K_0(H)$.  We are interested in decompositions of $H$-modules, rather than $H$-comodules, which are used in other contexts (e.g., in \cite{M} $K_0(H)$ refers to $H$-comodules).

For each of the Hopf algebras we consider, the irreducible representations are either one-dimensional or two-dimensional.  In what follows we shall use the following
notation:

\begin{notation} \label{not:reps}
Let $H$ be a Hopf algebra with a fixed set of algebra generators $x_1,\dots,x_e$.
Specifying a $d$-dimensional module is equivalent to providing a
$d \times d$ matrix for each generator of $H$ such that the matrices satisfy the
relations of $H$.

For one-dimensional representations, we let $T_{c_1,\dots,c_e}(t)$ be the $\kk$-vector
space with basis $t$ such that $x_it = c_it$ for $i = 1,\dots,e$.  For two-dimensional
representations, we let $\pi(u,v)$ be the $\kk$-vector space with basis $u$ and $v$
and denote by $\pi(x_i)$ the matrix that provides the action of $x_i$ on
$\kk u \oplus \kk v$.
\end{notation}

\begin{remark}
Let $H$ be a semisimple Hopf algebra.  In this paper, we search for AS regular algebras
on which $H$ acts (inner-faithfully) using the following procedure.  Let $A = T(V)/I$ be
a graded algebra generated in degree 1.  If $V$ is an (inner-faithful) $H$-module, we extend 
the action of $H$ to $T(V)$ using the coproduct of $H$.  This action passes to $A$ if and
only if $I$ is an $H$-submodule of $T(V)$.

In this paper we only study actions on quadratic algebras, hence we may also assume
$I \subseteq V \otimes V$.  Therefore the possible relations for algebras on which $H$ acts 
are governed by the decomposition of $V \otimes V$ into simple $H$-modules, and these are 
further restricted by those relations that give AS regular algebras.  Lemma \ref{lem:notAS} 
aids us in identifying algebras that are not AS regular in many of our examples.
\end{remark}

\subsection{Generating invariants}

The following lemmas are useful in finding minimal 
generating sets for some subrings of invariants.

\begin{lemma}
\label{binomial}
Let $x$ and $y$ be commuting elements of a ring,
and let $\ell \in \mathbb{N}$ with $\ell \geq 1$.
\begin{enumerate}
\item $x^\ell + y^\ell$ is  in the subalgebra $\Bbbk\langle x+y, xy \rangle$.
\item $x^\ell + (-1)^\ell y^\ell$ is  in the subalgebra $\Bbbk\langle x-y, xy \rangle$.
\end{enumerate}
\end{lemma}
\begin{proof}
The proofs are by induction on $\ell$.  We give the proof for (2), the proof for (1) being similar.

When $\ell$ is odd,
\begin{align*}
(x-y)^\ell - (x^\ell + (-1)^\ell y^\ell)
&= \sum_{m= 1}^{\ell -1} \binom{\ell}{m} x^m y^{\ell - m} \\
&= \sum_{m= 1}^{(\ell -1)/2} \binom{\ell}{m} [(-1)^{\ell-m}x^m y^{\ell -m} + (-1)^m x^{\ell -m} y^m] \\
&= \sum_{m= 1}^{(\ell -1)/2} \binom{\ell}{m} x^m y^m[(-1)^{\ell -m} y^{\ell - 2m} + (-1)^m x^{\ell -2m} ] \\
&= \sum_{m= 1}^{(\ell -1)/2} (-1)^m \binom{\ell}{m} x^m y^m[(-1)^{\ell -2m} y^{\ell - 2m} +  x^{\ell -2m} ],
\end{align*}
and the result follows by induction.
When $\ell$ is even the proof is similar, with an extra middle term $x^{\ell/2}y^{\ell/2} = (xy)^{\ell/2}$.  Hence the result holds in all cases.
\end{proof}

\begin{lemma} \label{GenInv}
Let $R$ be the following ring
\[
R=\frac{\kk\langle x,y,z,w\rangle}{(xy-yx,xz-zx,xw-wx,yz-zy,yw-wy,xy-\alpha z^{2k})}
\]
for some $\alpha\in\kk^*$ and some positive integer $k$. Then the element
\[
f_{t,l,s}(x,y,z,w)=z^t(x^l+(-1)^{t+s+l}y^l)w^s,
\]
where $t,l,s$ are nonnegative integers not all zero, is contained in the subalgebra generated by the elements
\[
z^2,x-y,(x+y)w,w^2,z(x+y),zw.
\]
\end{lemma}
\begin{proof}
If $s=t=0$ then the lemma follows from Lemma \ref{binomial}, therefore we can generate $f_{0,l,0}$ for all $l$. Now we prove that $f_{1,l,0}$ is generated by induction on $l$. If $l=1$ then the element $z(x+y)$ is one of the generators. Otherwise assume that the elements $f_{1,l-1,0},f_{1,l-2,0}$ have been constructed, then the equality 
\[
(x-y)z(x^{l-1}-(-1)^{l-1}y^{l-1})=z(x^l-(-1)^ly^l)-xyz(x^{l-2}-(-1)^{l-2}y^{l-1})
\]
shows that $f_{1,l,0}$ can be constructed using the claimed generators.

So far we have proved that the elements $f_{0,l,0}$ and $f_{1,l,0}$ are generated for all $l$. Now we show that all the elements of the form $f_{t,l,0}$ are generated. Indeed if $t=2j$ is even then $f_{t,l,0}=(z^2)^jf_{0,l,0}$ and if $t=2j+1$ is odd then $f_{t,l,0}=(z^2)^jf_{1,l,0}$.

Now we prove by induction that $f_{0,l,1}$ can be generated. If $l=1$ then $(x+y)w$ is one of the generators. Recall that the element $f_{0,l-1,0}$ has already been constructed and assume that $f_{0,l-2,1}$ has been constructed, then the equality
\[
f_{0,l-1,0}f_{0,1,1}=f_{0,l,1}+xyf_{0,l-2,1}
\]
shows that $f_{0,l,1}$ can be constructed.

To generate $f_{1,l,1}$ we observe that
\[
f_{1,l,1}=zwf_{0,l,0}.
\]
To show that $f_{t,l,1}$ is generated we just notice that if $t=2j$ is even then $f_{t,l,1}=(z^2)^jf_{0,l,1}$ and if $t=2j+1$ is odd then $f_{t,l,1}=(z^2)^jf_{1,l,1}$.

To conclude the proof we show that $f_{t,l,s}$ can be generated. If $s=2i$ is even then $f_{t,l,s}=f_{t,l,0}(w^2)^i$, if $s=2i+1$ is odd then $f_{t,l,s}=f_{t,l,1}(w^2)^i$.
\end{proof}

\begin{remark} \label{GenInvsEven}
The proof of the previous Lemma shows that if $s$ is even then $f_{t,l,s}$  is in the subalgebra  generated by $z^2,x-y,w^2,z(x+y)$.
\end{remark}

\begin{remark} \label{GenInvNoz}
Adopting the notation of Lemma \ref{GenInv}, we notice that in $R$ the element $(x^l+(-1)^{l+s}y^l)w^s$  is in the subalgebra generated by $z,x-y,(x+y)w,w^2$.
\end{remark}

\begin{remark} \label{GenInvNow}
Adopting the notation of Lemma \ref{GenInv}, we notice that in $R$ the element $z^t(x^l+(-1)^{l+t}y^l)$  is in the subalgebra generated by $z^2,x-y,z(x+y)$.
\end{remark}

\begin{lemma}\label{GenInvSkew}
In the ring $\kk_{-1}[x,y]$ the element
\[
g _{t,l}(x,y)=(xy)^{t}(x^l+(-1)^{t}y^l),
\]
 is in the subalgebra $\kk \langle x+y, xy(x-y) \rangle $.
\end{lemma}
\begin{proof}
We first notice that one can generate $x^2y^2$, indeed
\[
x^2y^2=-\frac{(x+y)xy(x-y)+xy(x-y)(x+y)}{4}.
\]
Therefore $g_{t,0}$ is generated for all even $t$ (and it is zero when $t$ is odd). 

Now we show by induction that $g_{0,l}$ and $g_{1,l}$ can be generated. We first check it for $l=1,2$, indeed 
\[
g_{0,1}=x+y,g_{0,2}=x^2+y^2=(x+y)^2,g_{1,1}=xy(x-y),g_{1,2}=xy(x+y)(x-y)-2x^2y^2.
\]
We now assume that $g_{0,m}$ and $g_{1,m}$ have been constructed for $m<l$. We consider the case $l$ even first and show how to generate $g_{0,l}=x^l+y^l$. Indeed
\begin{align*}
(x+y)^l&=(x^2+y^2)^{{l}/{2}}\\
&=\sum_{k=0}^{{l}/{2}}\binom{{l}/{2}}{k}x^{2k}y^{l-2k}\\
&=x^l+y^l+\sum_{k=1}^{{l}/{2}-1}\binom{{l}/{2}}{k}x^{2k}y^{l-2k}.
\end{align*}
We can assume that ${l}/{2}$ is odd, otherwise the only extra term in the previous summation is $x^{{l}/{2}}y^{{l}/{2}}=(x^2y^2)^{{l}/{4}}$, which we have already proved may be generated.
In this case
\begin{align*}
x^l+y^l&=(x+y)^l-\sum_{k=1}^{{(l-2)}/{4}}\binom{{l}/{2}}{k}(x^{2k}y^{l-2k}+x^{l-2k}y^{2k})\\
&=(x+y)^l-\sum_{k=1}^{{(l-2)}/{4}}\binom{{l}/{2}}{k}x^{2k}y^{2k}(x^{l-4k}+y^{l-4k})\\
&=(x+y)^l-\sum_{k=1}^{{(l-2)}/{4}}\binom{{l}/{2}}{k}x^{2k}y^{2k}g_{0,l-4k},
\end{align*}
and therefore $x^l+y^l$ can be generated. If $l$ is odd then
\[
x^l+y^l=(x+y)(x^{l-1}+y^{l-1})+g_{1,l-2},
\]
and we are done by induction.

Now we show that one can generate $g_{1,l}$. If $l$ is even then
\[
xy(x^l-y^l)=xy(x-y)g_{0,l-1}+x^2y^2g_{0,l-2}.
\]
If $l$ is odd then 
\[
xy(x^l-y^l)=xy(x-y)g_{0,l-1}-xyg_{1,l-2}.
\]
To see that one may obtain $g_{t,l}$ for $t \geq 2$, notice that
\[
g_{t,l}=
\begin{cases}
(-1)^i(x^2y^2)^ig_{0,l}\quad\mathrm{if}\;t=2i,\\
(-1)^i(x^2y^2)^ig_{1,l}\quad\mathrm{if}\;t=2i+1.
\end{cases}
\]
\end{proof}

%%%%%%%%%%%%%%%%%%%%%%%%%%%
%%  Section Kac-Palyutkin
%%%%%%%%%%%%%%%%%%%%%%%%%%%%%

\section{\texorpdfstring{The Kac-Palyutkin Hopf algebra $H_{8}$}{The Hopf algebra H(8}}
\label{sec:KP}

Let $H_8$ be the Kac-Palyutkin algebra of dimension 8;  it is the smallest dimensional semisimple Hopf algebra that is neither commutative nor cocommutative (nor is it a twist of such a Hopf algebra).  As an algebra, $H_8$ is generated by $x,y,$ and $z$ with relations
$$x^2 = y^2 = 1,\; xy = yx, \; zx = yz, \; zy = xz, \; z^2= \frac{1}{2}(1 + x + y -xy).$$
The coproduct $\Delta$ in $H_8$ given by
$$\Delta(x) = x \otimes x, \; \Delta(y) = y \otimes y,\; \Delta(z) = \frac{1}{2}( 1 \otimes 1 + 1 \otimes x + y \otimes 1 - y \otimes x) (z \otimes z)$$
while the counit $\epsilon$ is defined by
$$\epsilon(x) = 1, \; \epsilon(y) = 1, \; \epsilon(z) = 1$$
and the antipode $S$ is the anti-automorphism given by
$$S(x) = x, \; S(y) = y, \;  S(z) = z.$$
Using Notation \ref{not:reps}, there are four one-dimensional 
representations of $H_8$, namely:
$$T_{1,1,1}, \; T_{1,1,-1}, \; T_{-1,-1,\i}, \; T_{-1,-1,-\i}$$
There is a unique two-dimensional irreducible representation
$\pi(u,v)$ given by the matrices:
 $$\pi(x) = \begin{pmatrix}-1 & 0\\ 0 & 1 \end{pmatrix} \quad
   \pi(y) = \begin{pmatrix}1 & 0\\ 0 & -1\end{pmatrix}, \quad
   \pi(z) = \begin{pmatrix}0 & 1 \\ 1 & 0\end{pmatrix}.$$
Using the coproduct of $H_8$ we compute that under this action
$$z.u^2 = v^2, \;\; z.uv= -vu, \;\; z .vu = uv, \;\; z.v^2 = u^2$$
One can check that as an $H_8$-module,
\begin{align*}
\pi(u,v) \otimes \pi(u,v) =&~~T_{1,1,1}(u^2 + v^2) \oplus T_{1,1,-1}(u^2 - v^2) \oplus \\
                           &~~T_{-1,-1,-\i}(vu + \i uv) \oplus T_{-1,-1,\i}(vu - \i uv)
\end{align*}
so that all irreducible $H_8$-modules occur as direct summands of $\pi$ or
$\pi \otimes \pi$. Hence, $\pi$ is an inner-faithful $H$-module.
Moreover, the basis elements of the one-dimensional summands indicated above
give us the possible quadratic algebras on which $H_8$ acts.

In summary, we see that $H_8$ will act inner-faithfully on
$A = \kk  \langle u,v \rangle/(r)$ if $A_1 = \pi$ as an $H$-module,
and $r$ is any one of the four basis elements of the one-dimensional
$H_8$-modules occurring as a direct summand in $\pi \otimes \pi$ as listed above.
Moreover, in each case the algebra $A = \kk  \langle u,v \rangle/(r)$ is AS regular
of dimension 2. The table below gives the relation $r$ and the corresponding fixed ring $A^{H_8}$ in each of these cases.  One can check that in each case there is a copy of $T$ in the algebra $A$, so that the action $\pi$ is actually faithful.

\begin{center}
\vspace*{.1in}
 \begin{tabular}{|c|c|c|c|c|}
\hline
Case & $T$& Relation $r$ &  Fixed Ring $A^{H_8}$\\
\hline \hline (a) & $T_{1,1,1}$ & $u^2 + v^2$ & commutative hypersurface\\
\hline (b) &$T_{1,1,-1}$& $u^2 - v^2$ &  $\kk [u^2, (uv)^2 -(vu)^2]$\\
\hline (c)&$T_{-1,-1,-\i}$& $ vu + \i uv$    &  $\kk [u^2 + v^2,u^2v^2]$\\
\hline (d) &$T_{-1,-1,\i}$&  $vu -\i uv$  &  $\kk [u^2 + v^2,u^2v^2]$\\
\hline
\end{tabular}
~\\
~\\
\end{center}

Summarizing we have the following theorem.
\begin{theorem}
The Kac-Palyutkin Hopf algebra $H_8$ acts (inner-)faithfully on the AS regular algebras $\kk_{\pm \i}[u,v]$ and on
$\kk\langle u, v \rangle/(u^2 - v^2)$ with fixed subring a commutative polynomial ring, and hence $H_8$ is a reflection Hopf algebra for each of these three AS regular algebras of dimension two.
\end{theorem}
%%%%%%%%%%%%%%%%%%%%%%%%%%%%%%
%%  Section 3 Pansera H_2n^2
%%%%%%%%%%%%%%%%%%%%%%%%%%
\section{\texorpdfstring{The Hopf algebras $H_{2n^2}$ of Pansera}{The Hopf algebras H(2n2)}}

\label{sec:H2n2}

In \cite{P} D. Pansera defined an infinite family of semisimple Hopf algebras $H_{2n^2}$ of dimension $2n^2$ that act inner-faithfully on certain quantum polynomial algebras.  When $n=2$, this Hopf algebra is the $8$-dimensional semisimple algebra defined by Palyutkin \cite{KP}, which was discussed in the previous section.
We begin by reviewing the construction of these algebras.  Fix an integer $n \geq 2$.

Let  $G = \langle x \rangle \times \langle y \rangle$ be the direct product of two cyclic groups of order $n$ and let $q$ denote $e^{\frac{2\pi i}{n}}$,  a primitive $n$th root of unity. A complete set of orthogonal idempotents in the group algebra $\kk \langle x \rangle$
is given by $\{e_j \mid 0 \leq j \leq n-1\}$, where
$$e_j = \frac{1}{n} \sum_{i=0}^{n-1}q^{-ij}x^i;$$
similarly, we define $\overline{e_j} \in \kk [\langle y \rangle]$ by
$$\overline{e_j} = \frac{1}{n} \sum_{i=0}^{n-1}q^{-ij}y^i.$$
Let $\sigma$ denote  the automorphism of $\kk G$ given by $\sigma(x^i y^j)=x^j y^i$, and define the element $J \in \kk G \otimes \kk G$ by
$$J = \sum_{i=0}^{n-1} e_i \otimes y^i = \frac{1}{n} \sum_{i,j=0}^{n-1}q^{-ij}x^i\otimes y^j
= \sum_{i=0}^{n-1} x^i \otimes \overline{e_i}.$$
Note that $J$ is a right Drinfel'd twist of $\kk G$ (\cite[Lemma 2.10]{P}).
Letting $\mu$ denote the multiplication map on $\kk G$, one may show that $\mu(J)$ is 
invertible in $\kk G$, and that $\sigma(\mu(J)) = \mu(J)$.
Finally, define $H_{2n^2}$ as the factor ring of the skew polynomial extension of $\kk G$:$$H_{2n^2} = \frac{\kk G[z; \sigma]}{( z^2 - \mu(J) )}.$$
In \cite{P} it is shown that $H_{2n^2}$ is a Hopf algebra, with vector space basis $$\{x^iy^jz^k:\; 0\leq i,j \le n-1,\; 0\leq k \leq 1\},$$ where the Hopf structure of
$\kk G$ is extended to $H_{2n^2}$ by setting:
$$\Delta(z) = J( z \otimes z) = \frac{1}{n} \sum_{i,j=0}^{n-1} q^{-ij}x^iz \otimes y^j z,
$$$$\epsilon(z) = 1, \hspace{.2in} \text{ and } \hspace{.2in}
S(z) = z.$$

%%%%%%%%%%%%%%%%%%%%%%%%%%%%
% Subsection 3.2 Grothendieck ring $K_0(H_{2n^2})$
%%%%%%%%%%%%%%%%%%%%%%%%%%%%%%%%

\subsection{\texorpdfstring{The Grothendieck ring $K_0(H_{2n^2})$}{The Grothendieck ring K0(H(2n2))}}
We now fix a square root of $q$, and denote it by
$p = -e^{{\pi  \i}/{n}} = e^{{(n+1) \pi \i}/{n}}$. When $n$ is 
odd, $p$ is a primitive $n$th root of unity.  To give the irreducible 
representations of $H_{2n^2}$, we first record a lemma.

\begin{lemma}\label{lem:rootUnity}
Let $q$ be a primitive $n^\text{th}$ root of unity.  Then for all
$0 \leq i,j \leq n-1$, one has the equality
$$\left(\frac{1}{n}\sum_{w=0}^{n-1}\sum_{r=0}^{n-1}q^{-wr}q^{ir}q^{jw}\right) = q^{ij}.$$
\end{lemma}
\begin{proof}
The lemma follows from the following string of equalities:
\begin{align*}
\left(\frac{1}{n}\sum_{w=0}^{n-1}\sum_{r=0}^{n-1}q^{-wr}q^{ir}q^{jw}\right)
=\left(\frac{1}{n}\sum_{w=0}^{n-1}q^{jw}\sum_{r=0}^{n-1}(q^{i-w})^r\right)
=q^{ij},
\end{align*}
where the last equality follows from the fact that an $n^\text{th}$ root of unity
different from 1 is a root of the polynomial $1 + x + \dots + x^{n-1}$.
\end{proof}

Below, we again use Notation \ref{not:reps} to describe the representations of $H_{2n^2}$.

\begin{proposition}
Let $k = 0,\ldots,n-1$.  Then the  one-dimensional vector space $T^{\pm}_k$ with $H$-action
given by $T_{q^k,q^k,\pm p^{k^2}}$ is an $H$-module, and all one-dimensional $H$-modules are of this form.  Furthermore,
for all $0\leq i,j\leq n-1$, the two-dimensional vector space $\pi_{i,j}$ with
$H$-action given by
\[
\pi_{i,j}(x)=\begin{pmatrix} q^i & 0\\ 0 & q^j\end{pmatrix},\pi_{i,j}(y)=\begin{pmatrix} q^j & 0\\ 0 & q^i\end{pmatrix},\pi_{i,j}(z)=\begin{pmatrix}  0 & 1\\ q^{ij} & 0\end{pmatrix}.
\]
is an $H$-module. 
\end{proposition}

\begin{proof}
Let $T$ be a one-dimensional representation generated by $t$. The relations $x^n=1$ and 
$y^n=1$ tell us that $x$ and $y$ must act on $t$ as an $n$th root of unity. The 
relation $xz=zy$ tells us that $x$ and $y$ must act on $t$ as the same root of unity 
$q^k$.  Indeed, to understand the action of $z$ on $t$, Lemma \ref{lem:rootUnity} implies
that $z^2t = q^{k^2}t$, and hence $z$ must act on $t$ by multiplication by $\pm p^{k^2}$. 
This action obviously satisfies all the other relations.

Let $\pi_{i,j}$ be a two-dimensional vector space generated by elements $u$ and $v$ over 
which $H_{2n^2}$ acts as stated in the theorem. We prove that this action satisfies the 
relations of the algebra.   Again, the only relation that is not obviously satisfied is the 
relation involving $z^2$.

Again using Lemma \ref{lem:rootUnity}, on a vector of $\pi_{i,j}$ the element $\mu(J)$
acts as the matrix
\begin{equation*}
\frac{1}{n}\sum_{w=0}^{n-1}\sum_{r=0}^{n-1}q^{-wr}
\begin{pmatrix} q^i&0\\0&q^j\end{pmatrix}^r
\begin{pmatrix}q^j&0\\0&q^i\end{pmatrix}^w =
\begin{pmatrix} q^{ij}&0\\0&q^{ij}\end{pmatrix}.
\end{equation*}
It just remains to notice that 
$\begin{pmatrix}0&1\\q^{ij}&0\end{pmatrix}^2
=\begin{pmatrix} q^{ij}&0\\0&q^{ij}\end{pmatrix}$.
\end{proof}

\begin{remark} \label{rem:H2n2Reps}
Note that one has $\pi_{i,j} \cong \pi_{j,i}$, and that the representation $\pi_{i,i}$
is reducible.  Indeed, a straightforward computation shows
\begin{equation}\label{Pansera2dimlRed}
\pi_{i,i}(u,v)=T^+_i(u+p^{i^2}v)\oplus T^-_i(u-p^{i^2}v).
\end{equation}
We also abuse notation and read the subscripts of $\pi_{i,j}$ modulo $n$, since the action
of $H_{2n^2}$  depends only  on the subscripts modulo $n$.
\end{remark}

\begin{theorem} \label{RepH2n2}
The representations $T^\pm_k$ for $k = 0,\dots,n-1$ and $\pi_{i,j}$ for
$0\leq i < j\leq n-1$ are a complete set of irreducible representations of $H_{2n^2}$.
\end{theorem}

\begin{proof}
The fact that $\pi_{i,j}$ is irreducible follows from the fact that one may not
simultaneously diagonalize $\pi_{i,j}(x)$ and $\pi_{i,j}(z)$ when $i \neq j$.
Furthermore, the representations considered are distinct since the matrices 
$\pi_{i,j}(x)$ have different spectra.

Hence, we have found $2n$ one-dimensional representations and $\binom{n}{2}$ two-dimensional
irreducible representations of $H_{2n^2}$.  These are all the irreducible
representations since
\[
2n+4\binom{n}{2}=2n^2. \qedhere
\]
\end{proof}

\begin{theorem} \label{thm:k0h2n2}
The ring $K_0(H_{2n^2})$ has the fusion rules given below, where $0 \leq i,j,k,l \leq n-1$
and the subscripts on the right-hand side of each equality are to be read modulo $n$ as 
mentioned in Remark \ref{rem:H2n2Reps}.
\begin{equation}
\label{thm:k0h2n2Display}
\begin{split}
\pi_{i,j}(u,v)\otimes\pi_{k,l}(a,b)&=\pi_{i+l,j+k}(q^{kl+ki}ub,va)\oplus\pi_{i+k,j+l}(q^{il}ua,vb) \\
\pi_{i,j}(u,v)\otimes T^\pm_k(t)&=\pi_{i+k,j+k}(ut,vt) \\
T^\pm_k(t)\otimes\pi_{i,j}(u,v)&=\pi_{i+k,j+k}(tu,tv) \\
T^\pm_k(t)\otimes T^+_j(s)&=T^\pm_{k+j}(ts) \\
T^\pm_k(t)\otimes T^-_j(s)&=T^\mp_{k+j}(ts) \\
T^+_j(s)\otimes T^\pm_k(t)&=T^\pm_{j+k}(st) \\
T^-_j(s)\otimes T^\pm_k(t)&=T^\mp_{j+k}(st) \\
\end{split}
\end{equation}
\end{theorem}

\begin{proof}

A computation shows that the $H_{2n^2}$ action on $\pi_{i,j}(u,v) \otimes \pi_{k,l}(a,b)$ is as follows:
\begin{center}
\bgroup
\def\arraystretch{1.25}
\begin{tabular}{|c||c|c|c|c|}
\hline
     & $ua$             & $ub$          & $va$          & $vb$      \\ \hline \hline
$x$  & $q^{i+k}ua$      & $q^{i+l}ub$   & $q^{j+k}va$   & $q^{j+l}vb$ \\ \hline
$y$  & $q^{j+l}ua$      & $q^{j+k}ub$   & $q^{i+l}va$   & $q^{i+k}vb$ \\ \hline
$z$  & $q^{ij+kl+kj}vb$ & $q^{ij+jl}va$ & $q^{kl+ki}ub$ & $q^{il}ua$  \\ \hline
\end{tabular}
\egroup
\end{center}
The first two equalities in the statement of the theorem follow immediately from the previous table.
We show the computation for $z\cdot(vb)$:
\begin{align*}
z\cdot(vb)&=\frac{1}{n}\sum_{w=0}^{n-1}\sum_{r=0}^{n-1}q^{-wr}(x^rzv)(y^wzb)\\
&=\frac{1}{n}\sum_{w=0}^{n-1}\sum_{r=0}^{n-1}q^{-wr}(x^ru)(y^wa)\\
&=\frac{1}{n}\sum_{w=0}^{n-1}\sum_{r=0}^{n-1}q^{-wr}(q^{ir}u)(q^{lw}a)\\
&=q^{il}ua,
\end{align*}
where the fourth equality follows from Lemma \ref{lem:rootUnity}.  The other equalities follow 
similarly.
\end{proof}

\begin{proposition} 
The ring $K_0(H_{2n^2})$ is isomorphic to $K_0(\mathbb{Z}_n\wr S_2)$. 
\end{proposition}
\begin{proof}
A similar argument to the one used in the proof of Theorem \ref{RepH2n2} shows that all the one-dimensional representations of $\kk[\mathbb{Z}_n\wr S_2]$ are of the form
$U_k^\pm=U_{q^k,q^k,\pm 1}$ (where $U_{q^k,q^k,\pm 1}$ is defined in
a manner similar to $T_{q^k,q^k,\pm 1}$) and all the two-dimensional 
irreducible representation are of the form $\rho_{i,j}$ with
$0\leq i<j \leq n-1$ and
\[
\rho_{i,j}(x)=\begin{pmatrix} q^i & 0\\ 0 & q^j\end{pmatrix},\rho_{i,j}(y)=\begin{pmatrix} q^j & 0\\ 0 & q^i\end{pmatrix},\rho_{i,j}(z)=\begin{pmatrix}  0 & 1\\ 1 & 0\end{pmatrix}.
\]
Now it is just a matter of proving that the multiplication in
$K_0(\mathbb{Z}_n\wr S_2)$ is the same as in $K_0(H_{2n^2})$. A computation shows 
that the $\mathbb{Z}_n\wr S_2$ action on $\rho_{i,j}(u,v)\otimes\rho_{l,k}(a,b)$ is 
as follows 
\begin{center}
\bgroup
\def\arraystretch{1.25}
\begin{tabular}{|c||c|c|c|c|}
\hline
     & $ua$             & $ub$          & $va$          & $vb$      \\ \hline \hline
$x$  & $q^{i+l}ua$      & $q^{i+k}ub$   & $q^{j+l}va$   & $q^{j+k}vb$ \\ \hline
$y$  & $q^{j+k}ua$      & $q^{j+l}ub$   & $q^{i+k}va$   & $q^{i+l}vb$ \\ \hline
$z$  & $vb$ & $va$ & $ub$ & $ua$  \\ \hline
\end{tabular}
\egroup
\end{center}
which gives
\[
\rho_{i,j}(u,v)\otimes\rho_{k,l}(a,b)=\rho_{i+l,j+k}(ub,va)\oplus\rho_{i+k,j+l}(ua,vb).
\]
The other products are checked similarly.
\end{proof}

\begin{theorem} \label{thm:H2n2IF}
If $A$ is an algebra generated by $u,v$ in degree 1 where $\kk u\oplus \kk v=\pi_{i,j}(u,v)$ then the action of $H_{2n^2}$ on $A$ is inner-faithful if and only if $(i^2-j^2,n)=1$ { for $0 \leq i < j \leq n-1$}.
\end{theorem}
\begin{proof}
Let $I$ be a Hopf ideal of $H_{2n^2}$ such that $IA=0$. We denote by $R$ the group algebra $\kk [\mathbb{Z}_n\times\mathbb{Z}_n]$, and think of it as a Hopf subalgebra of $H_{2n^2}$ generated by $x$ and $y$. Then $I\cap R$ is a Hopf ideal of $R$. By \cite[Lemma 1.4]{P} there is $N$ a normal subgroup of $\mathbb{Z}_n\times\mathbb{Z}_n$ such that 
\begin{equation} \label{aug}
I\cap R=R \kk [N]^+. 
\end{equation}
Since $I\cap R\neq0$ there is $(s,t)\neq(0,0)$ such that $x^sy^t\in N$. By \ref{aug} $1-x^sy^t\in I$, hence
\[
(1-q^{is+tj})u=(1-x^2y^t)u=0
\]
\[
(1-q^{js+ti})v=(1-x^2y^t)v=0.
\]
This implies that 
\[
\begin{cases}
is+tj\equiv 0 \pmod{n} \\
js+ti\equiv 0 \pmod{n}
\end{cases}
\]
i.e., the vector $\begin{pmatrix} s\\ t\end{pmatrix}$ is in the kernel of the matrix 
\[
M=\begin{pmatrix} i & j\\ j& i\end{pmatrix}\in M_{2\times2}(\mathbb{Z}_n).
\]
If $(i^2-j^2,n)=1$ then $M$ is injective
hence $s\equiv 0\pmod{n}$ and $t\equiv 0 \pmod{n}$, which implies
$I\cap R=0$. By \cite[Lemma 2.12]{P}, $I=0$ which means that the action 
is inner-faithful. If $M$ is not injective then consider a nonzero vector 
$v=\begin{pmatrix} s \\ t\end{pmatrix}$ such that $Mv=0$ then $(1-
x^sy^t)A=0$ and so the action is not inner-faithful. 
\end{proof}

%%%%%%%%%%%%%%%%%%%%%%%%%%%%%%%%
% Subsection 3.2 Inner-faithful actions of H_{2n^2} on A regular
%%%%%%%%%%%%%%%%%%%%%%%%%%%%%

\subsection{\texorpdfstring{Inner-faithful Hopf actions of $H_{2n^2}$ on AS regular algebras and their fixed rings}{Hopf actions of H(2n2) on AS regular algebras and their fixed rings}} 

Recall that $q = e^{{2 \pi \i}/{n}}$ and $p = -e^{{\pi \i}/{n}} = e^{{(n+1) \pi \i/}{n}}$ (which is a primitive $n$th root of unity when $n$ is odd).  {Let $H_{2n^2}$ act on a quadratic algebra with two generators  and $A_1 = \pi_{i,j}$ for $i\not\equiv j\;(\mathrm{mod}\;n)$.  Then by Theorem \ref{thm:k0h2n2}
$$\pi_{i,j}(u,v) \otimes \pi_{i,j}(u,v) = \pi_{i+j,i+j}(q^{ij+i^2} uv, vu) \oplus 
\pi_{2i,2j}(q^{ij}u^2, v^2)$$
Then by Remark \ref{rem:H2n2Reps} $\pi_{i+j,i+j}$ decomposes into one-dimensional representations whose basis element could be taken to be the relation in an algebra $A$  that $H_{2n^2}$ acts upon, namely 
$$\pi_{i+j,i+j}(q^{ij+i^2} uv, vu)  = T^{+}_{i+j}(q^{ij+i^2} uv+ p^{(i+j)^2}vu)\oplus T_{i+j}^{-}(q^{ij+i^2} uv- p^{(i+j)^2}vu).$$
Using either of these basis elements as the relation in $A$ and recalling that $p^2 = q$ shows that
} the Hopf algebra $H_{2n^2}$ acts on the {AS regular algebras of dimension 2}
\[
A^\pm=\frac{\kk\langle u,v\rangle}{(p^{i^2-j^2}uv\pm vu)},\quad\mathrm{where}\;\kk u\oplus \kk v=\pi_{ij}(u,v),\; i\not\equiv j\;(\mathrm{mod}\;n).
\]

\begin{remark}
When $n$ is even the representation $\pi_{2i,2j}$ decomposes as a sum of two representations of dimension one if $i\equiv j\pmod{({n}/{2})}$, giving extra choices for a relation in $A$. But by Theorem \ref{thm:H2n2IF} these actions are inner-faithful if and only if $(i^2-j^2,n)=1$. The conditions that $n$ is even, $i\equiv j \pmod{({n}/{2})}$, and $(i^2-j^2,n)=1$ can be simultaneously true only if $n=2$. Therefore, the ``extra'' inner-faithful algebra actions can occur only if $n=2$, which means the algebra $H_{2n^2}$ is the Kac-Palyutkin algebra, which was analyzed in Section \ref{sec:KP}.
\end{remark}

We first consider the algebra $A^-$, which we denote by $A$. We want to compute the fixed ring $A^{H_{2n^2}}$ and determine when it is AS regular.
\begin{lemma}
The $z$ action on the monomials of $A=A^-$ is given by

\begin{equation}\label{uvPow}
z\cdot (u^av^b)=q^{\left[a+\binom{a}{2}+\binom{b}{2} \right]ij+abj^2} p^{ab(i^2-j^2)}u^bv^a.
\end{equation}
\end{lemma}

\begin{proof}
We first prove that
\begin{equation} \label{uPow}
z\cdot u^m=q^{(m+\binom{m}{2})ij}v^m.
\end{equation}
It is clear if $m=1$. We assume it is true for $m$ and prove it for $m+1$:
\begin{align*}
z\cdot u^{m+1}&=z\cdot(u\cdot u^m)\\
&=\frac{1}{n}\sum_{w=0}^{n-1}\sum_{r=0}^{n-1}q^{-wr}(x^rzu)(y^wzu^m)\\
&=\frac{q^{(m+1+\binom{m}{2})ij}}{n}\sum_{w=0}^{n-1}\sum_{r=0}^{n-1}q^{-wr}(x^rv)(y^wv^m)\\
&=\frac{q^{(m+1+\binom{m}{2})ij}}{n}\sum_{w=0}^{n-1}\sum_{r=0}^{n-1}q^{-wr}q^{jr}q^{iwm}v^{m+1}\\
&=q^{(m+1+\binom{m+1}{2})ij}v^{m+1},
\end{align*}
where the last equality follows from Lemma \ref{lem:rootUnity}.  
Similarly one proves
\begin{equation}\label{vPow}
z\cdot v^m=q^{\binom{m}{2}ij}u^m.
\end{equation}

To establish the result, note that Lemma \ref{lem:rootUnity} again implies
\begin{equation} \label{qPow}
\sum_{w=0}^{n-1} \sum_{r=0}^{n-1}  q^{-j^2ab -wr + rja + wjb} =
q^{-j^2ab} \sum_{w=0}^{n-1} \sum_{r=0}^{n-1}  q^{-wr}q^{rja}q^ {wjb} = n.
\end{equation} 
Equation \eqref{uvPow} now follows from \eqref{uPow}, \eqref{vPow}, 
\eqref{qPow}, the defining relation of $A$, and
\begin{align*}
z\cdot(u^av^b)&=\frac{1}{n}\sum_{w=0}^{n-1}\sum_{r=0}^{n-1}q^{-wr}(x^rzu^a)(y^wzv^b).
\qedhere
\end{align*}
\end{proof}

In particular, one has
\[
z\cdot u^n=\begin{cases} v^n & n\;\mathrm{odd}\\ q^{\binom{n}{2}ij}v^n = (-1)^{ij}v^n& n\;\mathrm{even}   \end{cases},\;\;\;
z\cdot v^n=\begin{cases} u^n & n\;\mathrm{odd}\\q^{\binom{n}{2}ij}u^n = (-1)^{ij}u^n& n\;\mathrm{even}   \end{cases}
\]
and $z\cdot (u^nv^n)= u^nv^n$ for all $n$. This implies that $u^nv^n$ is a fixed element for all $n$, $u^n+v^n$ is fixed when $n$ is odd and $u^n+(-1)^{ij}v^n$ is fixed when $n$ is even. 
\begin{theorem} \label{thm:H2n2refhopf}
If $H = H_{2n^2}$ acts inner-faithfully by $\pi_{ij}$  (i.e., if $(i^2-j^2,n)=1$ { for $0 \leq i < j \leq n-1$}) on
\[A=A^-=\frac{\kk\langle u,v\rangle}{(p^{i^2-j^2}uv - vu)},
\]
where $p = -e^{{\pi \i}/{n}} = e^{{(n+1) \pi \i}/{n}}$, then the fixed subring is
\[
A^{H}=\begin{cases} \kk[u^n+v^n,u^nv^n] & n\;\mathrm{odd}\\ \kk[u^n+(-1)^{ij}v^n,u^nv^n] & n\;\mathrm{even}.     \end{cases}
\]
Hence $H_{2n^2}$ is a reflection Hopf algebra { under this action} for $A^-$ when $(i^2-j^2,n)=1$.
\end{theorem}
\begin{proof}
Let $\displaystyle\sum_{a,b}\alpha_{a,b}u^av^b$ be an element of $A$ fixed by the action of $H_{2n^2}$. Then
\[
x\cdot\sum_{a,b}\alpha_{a,b}u^av^b=\sum_{a,b}q^{ia+jb}\alpha_{a,b}u^av^b,
\]
and
\[
y\cdot\sum_{a,b}\alpha_{a,b}u^av^b=\sum_{a,b}q^{ja+ib}\alpha_{a,b}u^av^b,
\]
so in order for $\displaystyle\sum_{a,b}\alpha_{a,b}u^av^b$ to be fixed by the action of $x$ and $y$ we must have
\[
\begin{cases}
ia+jb\equiv0\pmod{n}\\
ja+ib\equiv0\pmod{n}
\end{cases}
\]
which is equivalent to
\[
\begin{pmatrix} a\\b\end{pmatrix}\in\mathrm{ker}\begin{pmatrix} i& j\\j&i\end{pmatrix},
\]
but this kernel is zero because $(i^2-j^2,n)=1$, hence $a,b\equiv0 \pmod{n}$.

Now we assume $n$ odd, the case $n$ even is similarly proved. A fixed element must be of the form
\begin{equation}\label{FixedEl}
\sum_{a,b}\alpha_{a,b}u^{an}v^{bn}
\end{equation}
and by using \eqref{uvPow} we get
\[
z\cdot\sum_{a,b}\alpha_{a,b}u^{an}v^{bn}=\sum_{a,b}\alpha_{a,b}u^{bn}v^{an},
\]
because
\[
z\cdot(u^{an}v^{bn})=q^{\left[an + \binom{an}{2}+\binom{bn}{2} \right]ij + ab n^2 j^2}p^{abn^2(i^2-j^2)}u^{bn}v^{an} = u^{bn}v^{an}.
\]
In order for an element of the form \eqref{FixedEl} to be fixed by the action of $z$ we must have $\alpha_{a,b}=\alpha_{b,a}$ for all $a,b$. Hence a fixed element has the form
\[
\sum_{a\leq b}\alpha_{a,b}(u^{an}v^{bn}+u^{bn}v^{an}).
\]
It follows from Lemma \ref{binomial}, by setting $x=u^n$ and $y=v^n$, (since $n$ is odd $p^n =1$ so $x$ and $y$ commute), that this invariant ring is generated by $u^n+v^n,u^nv^n$.

It remains only to prove that the generators of the fixed ring are algebraically independent. This follows because they form a regular sequence in the commutative Cohen-Macaulay ring $\kk [u^n,v^n]$.
\end{proof}

When $A=A^+$ one checks similarly that the action of $z$ on monomials in $A$ is as given in the following lemma.
\begin{lemma} \label{lem:za+}
The action of $z$ on monomials of $A=A^+$ is given by
\[{z\cdot (u^av^b)=(-1)^{ab}q^{\left[a+\binom{a}{2}+\binom{b}{2} \right]ij+abj^2} p^{ab(i^2-j^2)}u^bv^a.}
\]
\end{lemma}

When $n$ is even the computations above and Lemma \ref{lem:za+} show that $A^H$ is the 
same for $A=A^+$ as for $A=A^-$.  Using an argument similar to the even case, one can
show that when $n$ is odd an invariant must have the form
\[
\sum_{a\leq b}\alpha_{a,b}(u^{an}v^{bn}+(-1)^{abn^2}u^{bn}v^{an}).
\]
We may rewrite the previous expression as
\[
\sum_{k,p}\beta_{k,p}(u^nv^n)^p((v^{n})^{k}+(-1)^{p}(u^{n})^{k}).
\]
Applying Lemma \ref{GenInvSkew} with $x=v^n$ and $y=u^n$ shows that $u^n + v^n$
and $u^nv^n(u^n - v^n)$ generate the invariants. The subring of $\kk_{-1}[x,y]$ generated by $x+y$ and $xy(x-y)$ is the subring invariant under the transposition $x\leftrightarrow y$ \cite[Example 3.1]{KKZ2}, which is not AS regular \cite[Theorem 1.5(2)]{KKZ2}.

We summarize these cases in the following theorem. 

\begin{theorem} \label{thm:H2n2fixed+}
When $H= H_{2n^2}$ acts { by $\pi_{ij}$} on
\[A=A^+ = \frac{\kk\langle u,v\rangle}{(p^{i^2-j^2}uv + vu)}
\]
for $p = -e^{{\pi \i}/{n}} = e^{{(n+1) \pi \i}/{n}}$, 
inner-faithfully  (i.e., $(i^2-j^2,n)=1$ { for $0 \leq i < j \leq n-1$}) then
\begin{enumerate}
\item when $n$ is even, the invariant ring is
\[
A^{H}= \kk[u^n+(-1)^{ij}v^n,u^nv^n],\]
and $H_{2n^2}$ is a reflection Hopf algebra for $A$,
\item when $n$ is odd, the invariant ring is
\[
A^{H} = \kk\langle u^n+v^n,u^nv^n(u^n-v^n)\rangle,
\]
which is not AS regular.
\end{enumerate}
\end{theorem}

%%%%%%%%%%%%%%%%%%%%%%%%%%%%%%%%%%%
%% Section 4:   B_4m
%%%%%%%%%%%%%%%%%%%%%%%%%%%%%%%%%%%%%%

\section{\texorpdfstring{The Hopf algebras $\cB_{4m}$ of Masuoka}{The Hopf algebras B(4m) of Masuoka}}
\label{sec:B4m}

The $4m$-dimensional Hopf algebras $\cA_{4m}$ and $\cB_{4m}$ for $m \geq 2$ were defined by Masuoka in \cite[Definition 3.3]{M}. 
Note that $\cB_{8}$ is the Kac-Palyutkin algebra considered in Section \ref{sec:KP}.
Let   $K = \kk\langle a \rangle$ be the group algebra of a cyclic group of order 2; $K$ is identified with its dual $\kk^{\langle a \rangle}$, in which $e_0$ and $e_1$ are the idempotents
$e_0 =(1+a)/2$ and $e_1 = (1-a)/2$. The Hopf algebras $\cA_{4m}$ and $\cB_{4m}$ are defined as algebras over $K$, with $K$ a central Hopf subalgebra (with $a$ group-like). The Hopf algebra $\cA_{4m}$ is generated as an algebra over $K$ by the two elements $s_+$ and $s_-$ with the relations:
$$s_{\pm}^2 = 1, \;\; (s_+s_-)^m = 1.$$
The coproduct, counit and antipode in $\cA_{4m}$ are:
$$\Delta(s_{\pm}) = s_{\pm}  \otimes e_0s_{\pm} + s_{\mp} \otimes e_1 s_\pm,\;\; \epsilon(s_\pm) = 1,\;\; S(s_\pm) = e_0s_\pm + e_1 s_\mp.$$
$\cB_{4m}$  is defined in the same way, except the relation $(s_+s_-)^m = 1$ is
replaced by the relation $(s_+s_-)^m = a$.

Next we compute the Grothendieck rings of the irreducible modules of $\cB_{4m}$.  
We note that the Grothendieck rings given in \cite{M} are for the irreducible comodules.
The irreducible modules for these algebras are all one-dimensional or two-dimensional.  

%%%%%%%%%%%%%%%%%%%%%%%%
%Subsection 4.1  The Grothendieck ring $K_0(\cB_{4m})$
%%%%%%%%%%%%%%%%%%%%%%%%%%%5

\subsection{\texorpdfstring{The Grothendieck ring $K_0(\cB_{4m})$}{The Grothendieck ring K0(B(4m))}}

The next proposition is straightforward.

\begin{proposition} \label{prop:repsB4m}
 The one-dimensional representations of $\cB_{4m}$ are of the form $T_{1,1,1},T_{1,-1,(-1)^m},T_{-1,1,(-1)^m},T_{-1,-1,1}$. The irreducible two-dimensional  representations are 
 \[
 \pi_i(s_+)=\begin{pmatrix} 0&1\\1&0\end{pmatrix},\pi_i(s_-)=\begin{pmatrix}0&\lambda^{-i}\\\lambda^{i}&0\end{pmatrix},\pi_i(a)=\begin{pmatrix}(-1)^i&0\\0&(-1)^i\end{pmatrix}
 \]
 where $\lambda = e^{\pi \i/m}$ is a primitive $2m$th root of unity, $i=1,\ldots, m-1$. 
\end{proposition}

\begin{notation}
We denote by $\pi_0$ and $\pi_m$ the (reducible) two-dimensional representations given by
 \begin{align*}
 \pi_0(s_+)=\begin{pmatrix} 0&1\\1&0\end{pmatrix},&\quad\pi_0(s_-)=\begin{pmatrix}0&1\\1&0\end{pmatrix},&&\pi_0(a)=\begin{pmatrix}1&0\\0&1\end{pmatrix},\\
 \pi_m(s_+)=\begin{pmatrix} 0&1\\1&0\end{pmatrix},&\quad\pi_m(s_-)=\begin{pmatrix}0&-1\\-1&0\end{pmatrix},&&\pi_m(a)=\begin{pmatrix}(-1)^m&0\\0&(-1)^m\end{pmatrix}.
 \end{align*}
 A straightforward computation shows
 \begin{align*}
 \pi_0(u,v)=&\;\;T_{1,1,1}(u+v)\oplus T_{-1,-1,1}(u-v),\\
 \pi_m(u,v)=&\;\;T_{1,-1,(-1)^m}(u+v)\oplus T_{-1,1,(-1)^m}(u-v).
 \end{align*}
\end{notation}
The next two results are also straightforward.
\begin{theorem} \label{K0B4m}
The Grothendieck ring $K_0(\cB_{4m})$ has the following fusion rules:
\[
\pi_i(u,v)\otimes\pi_j(x,y)=
\begin{cases}
\pi_{i-j}(uy,vx)\oplus\pi_{i+j}(ux,vy) & i+j\leq m-1, i-j\geq0\\
\pi_{j-i}(vx,uy)\oplus\pi_{i+j}(ux,vy) & i+j\leq m-1, i-j\leq0\\
\pi_{i-j}(uy,vx)\oplus\pi_{2m-i-j}(vy,ux) & i+j>m-1, i-j\geq0\\
\pi_{j-i}(vx,uy)\oplus\pi_{2m-i-j}(vy,ux) & i+j>m-1, i-j\leq0
\end{cases},
\]
when $j$ is even, and
\[
\pi_i(u,v)\otimes\pi_j(x,y)=
\begin{cases}
\pi_{i-j}(\lambda^ivy,ux)\oplus\pi_{i+j}(\lambda^ivx,uy) & i+j\leq m-1, i-j\geq0\\
\pi_{j-i}(ux,\lambda^ivy)\oplus\pi_{i+j}(\lambda^ivx,uy) & i+j\leq m-1, i-j\leq0\\
\pi_{i-j}(\lambda^ivy,ux)\oplus\pi_{2m-i-j}(uy,\lambda^ivx) & i+j>m-1, i-j\geq0\\
\pi_{j-i}(ux,\lambda^ivy)\oplus\pi_{2m-i-j}(uy,\lambda^ivx) & i+j>m-1, i-j\leq0
\end{cases},
\]
when $j$ is odd. In addition
\[
T_{\pm1,\pm1,1}(t)\otimes\pi_i(u,v)=\pi_i(tu,\pm tv),
\]
\[
\pi_i(u,v)\otimes T_{\pm1,\pm1,1}(t)=\pi_i(ut,\pm vt),
\]
\[
T_{\pm1,\mp1,(-1)^m}(t)\otimes\pi_i(u,v)=\pi_{m-i}(\pm(-1)^itv,tu),
\]
\[
\pi_i(u,v)\otimes T_{\pm1,\mp1,(-1)^m}=\pi_{m-i}(\pm(-1)^mtv,tu),
\]
\[
T_{\delta,-\delta,(-1)^m}(t)\otimes T_{\varepsilon,\varepsilon,1}(s)=T_{\delta\varepsilon,-\delta\varepsilon,(-1)^m}(ts),\quad \delta=\pm1,\varepsilon=\pm1,
\]
\[
T_{\varepsilon,\varepsilon,1}(s)\otimes T_{\delta,-\delta,(-1)^m}(t)= T_{\varepsilon\delta,-\varepsilon\delta,(-1)^m}(st),\quad \delta=\pm1,\varepsilon=\pm1. 
\]
\end{theorem}

We notice that the dihedral group $D_{4m}$ has the same representations as $\cB_{4m}$. Comparing the fusion rules of the two Grothendieck rings we notice that there is an isomorphism between them defined as $\pi_i\mapsto\pi_i$ and $T_{\alpha,\beta\gamma}\mapsto T_{\alpha,\beta,\gamma}$. We have proved the following

\begin{proposition}\label{K0B4m=K0G}
The ring $K_0(\cB_{4m})$ is isomorphic to $K_0(D_{4m})$.
\end{proposition}

We use the previous proposition to analyze the inner-faithful representations of $\cB_{4m}$ by reducing the problem to analyzing the inner-faithful representations of $D_{4m}$.

\begin{theorem} \label{thm:B4mIF}
Let $A$ be a $\kk$-algebra generated in degree 1 by $u$ and $v$ with $\kk u\oplus \kk v=\pi_i(u,v)$. The action of $\cB_{4m}$ on $A$ is inner-faithful if and only if $(i,2m)=1$ for $i = 1 \ldots, m-1$.
\end{theorem}

\begin{proof}
By Proposition \ref{K0B4m=K0G}, a representation $\pi_i$ generates (see Definition \ref{GenRep}) $K_0(\cB_{4m})$ if and only if it generates $K_0(D_{4m})$, so from now on we will work with the group algebra $\kk D_{4m}$ and $B$ will be a $\kk$-algebra generated in degree 1 by $u$ and $v$ with $\kk u\oplus \kk v=\pi_i(u,v)$ over which $D_{4m}$ acts.

Let $I$ be a Hopf ideal of $\kk D_{4m}$ such that $IA=0$ with $I\neq0$. Then by \cite[Lemma 1.4]{P} there is a normal subgroup $N$ of $D_{4m}$ such that $I=(\kk D_{4m})(\kk N)^+$. Since $I$ is not trivial then neither is $N$, hence there is a element in $N$ of the form
\[
s_-(s_+s_-)^p\in N\quad p=0,\ldots, 2m-1
\]
or
\[
(s_+s_-)^p\in N\quad p=1,\ldots,2m-1.
\]

We first deal with the case $s_-(s_+s_-)^p\in N$. In this case $1-s_-(s_+s_-)^p\in I$ and hence it annihilates $A$. But this element acts on $A$ as the matrix
\[
\begin{pmatrix}
1&-\lambda^{-i(p+1)}\\
-\lambda^{i(p+1)}&1
\end{pmatrix}
\]
which is never zero, a contradiction. Hence all the nontrivial elements in $N$ must be of the form $(s_+s_-)^p$ with $p=1,\ldots,2m-1$. As a result there is an element in $I$ of the form $1-(s_+s_-)^p$ with $p=1,\ldots,2m-1$. The element $1-(s_+s_-)^p$ acts on $A$ as
\[
\begin{pmatrix}
1-\lambda^{ip}&0\\
0&1-\lambda^{-ip}
\end{pmatrix}.
\]
This matrix is zero if and only if $\lambda^{ip}=1$ if and only if $ip\equiv 0\pmod{2m}$ for some $p=1,\ldots,2m-1$ which is equivalent to $(2m,i)\neq1$. 

Hence if the action is not inner-faithful then $(i,2m)\neq1$. If $(i,2m)\neq1$ then choose $p$ between 1 and $2m-1$ such that $ip$ is a multiple of $2m$. The Hopf ideal generated by $1-(s_+s_-)^p$ annihilates $A$ and hence the action is not inner-faithful.

\end{proof}

%%%%%%%%%%%%%%%%%%%%%%%%%%%%%5
% Subsection 4.2 Inner-faithful Hopf actions of \texorpdfstring{$\cB_{4m}$ on AS regular
%%%%%%%%%%%%%%%%%%%%%%%%%%

\subsection{Inner-faithful Hopf actions of \texorpdfstring{$\cB_{4m}$ on AS regular algebras and their fixed rings}{Hopf actions of B(4m)}}
{Noting when $\pi_i(u,v) \otimes \pi_i(u,v)$ has one-dimensional summands}, by Theorems \ref{K0B4m} and \ref{thm:B4mIF} the Hopf algebra $\cB_{4m}$ acts inner-faithfully on the AS regular algebras {of dimension 2} { for $i = 1 \ldots, m-1$,}
\[
A^{\pm}=\frac{\kk\langle u,v\rangle}{(u^2\pm\lambda^iv^2)},\quad \text{ where } \kk u\oplus \kk v=\pi_i(u,v), \text{ and } \;(i,2m)=1,
\]
for  $\lambda = e^{{\pi \i}/{m}}$, a primitive $2m$th root of unity.
We first set $A=A^-$, $H= \cB_{4m}$, and calculate $A^{H}$.\\

\begin{remark}
By Theorem \ref{K0B4m} If $m$ is even and $i=j={m}/{2}$ the representation $\pi_{2m-i-j}$ decomposes, giving extra algebra actions. But by Theorem \ref{thm:B4mIF} these actions are inner-faithful if and only if $(i,2m)=1$. But $({m}/{2},2m)=1$ if and only if $m=2$. So the only ``extra'' inner-faithful algebra action occurs when $\cB_{4m}\cong H_8$, this algebra was analyzed in Section \ref{sec:KP}.

\end{remark}

\begin{theorem}
\label{B4mfixed}
If $H= \cB_{4m}$ acts on $A=A^-$ inner-faithfully { by $\pi_i$ (i.e. $(i, 2m) = 1$ for $i = 1 \ldots, m-1$)}, then
\[
A^{H}=\kk[u^2,(uv)^m-(vu)^m],
\]
and hence $H= \cB_{4m}$ is a reflection Hopf algebra for $A=A^-$.
\end{theorem}
\begin{proof}
It is easy to check that $u^2$ and $(uv)^m-(vu)^m$ are fixed.

If an element is fixed by $a$ then it must be of even degree, hence, by Lemma \ref{lem:u2v2Basis},
it is of the form
\[
F=\sum_{p,q}\alpha_{p,q}u^{2p}(vu)^q+\sum_{p,q}\beta_{p,q}u^{2p+1}(vu)^{q-1}v.
\]
A computation shows 
\begin{align*}
s_+(u^{2p}(vu)^q)&=\lambda^{-iq}u^{2p+1}(vu)^{q-1}v\\
s_+(u^{2p+1}(vu)^{q-1}v)&=\lambda^{iq}u^{2p}(vu)^q\\
s_-(u^{2p}(vu)^q)&=\lambda^{iq}u^{2p+1}(vu)^{q-1}v\\
s_-(u^{2p+1}(vu)^{q-1}v)&=\lambda^{-iq}u^{2p}(vu)^q.
\end{align*}
Setting $s_+F=F$ yields $\alpha_{p,q}=\beta_{p,q}\lambda^{iq}$ and setting $s_-F=F$ yields $\alpha_{p,q}=\beta_{p,q}\lambda^{-iq}$. Both identities, combined, lead to $q\equiv0 \pmod{m}$. Hence setting $q=km$ and using the fact that $i$ is odd we deduce that $F$ must have the form
\[
F=\sum_{p,q}\beta_{p,q}u^{2p}((uv)^{km}+(-1)^k(vu)^{km}).
\]
This element is generated by the claimed elements by Lemma \ref{binomial} since $u^2$ is central in $A$.

It remains only to prove that the generators of the fixed ring are 
algebraically independent.  
Let $X = u^2$, $Y = (uv)^m$ and $Z = (vu)^m$.  Then the algebra
$\kk [u^2,(uv)^m,(vu)^m]$ is isomorphic to 
$\kk[X,Y,Z]/(X^{2m}+YZ)$ since the latter algebra is a commutative domain (as the
element $X^{2m}+YZ$ is irreducible by Eisenstein's criterion), and the former
algebra has GK dimension two.  It follows that $\kk[u^2,(uv)^m,(vu)^m]$ is
Cohen-Macaulay.  Since $u^2$ and $(uv)^m - (vu)^m$ form 
regular sequence in a Cohen-Macaulay algebra, they are algebraically independent.
\end{proof}
Next we consider $A = A^+$, and show that $H = \cB_{4m}$ is not a reflection Hopf algebra for $A^+$.
\begin{theorem}\label{thm:B4mPlus}
The fixed ring for the inner-faithful action of $\cB_{4m}$ on $A^+$ { by $\pi_i$ (i.e. $(i, 2m) = 1$ for $i = 1 \ldots, m-1$)} is 
\[
(A^+)^H=\kk\langle u^4,(uv)^m-(vu)^m,u^2((uv)^m+(vu)^m)\rangle.
\]
Furthermore, this ring is not AS regular.
\end{theorem}
\begin{proof}The argument is similar to the proof of Theorem~\ref{B4mfixed}. Any invariant $F$ must have even degree so
\[
F=\sum_{p,q}\alpha_{p,q}u^{2p}(vu)^q+\sum_{p,q}\beta_{p,q}u^{2p+1}(vu)^{q-1}v.
\]

The action of $\cB_{4m}$ is 
\begin{align*}
s_+(u^{2p}(vu)^q)&= (-1)^p\lambda^{-iq}u^{2p+1}(vu)^{q-1}v\\
s_+(u^{2p+1}(vu)^{q-1}v)&= (-1)^p\lambda^{iq}u^{2p}(vu)^q\\
s_-(u^{2p}(vu)^q)&= (-1)^p\lambda^{iq}u^{2p+1}(vu)^{q-1}v\\
s_-(u^{2p+1}(vu)^{q-1}v)&= (-1)^p\lambda^{-iq}u^{2p}(vu)^q.
\end{align*}
As before, this implies that $q = km$ for some integer $k$. Thus,
\[ F = \sum_{p, k} \beta_{p,k} u^{2p} ( (uv)^{km} + (-1)^{p+k} (vu)^{km}).
\]
Using Remark \ref{GenInvNow} by setting $z=u^2,x=(uv)^m,y=(vu)^m$, we see that the
generators above indeed generate the fixed ring.
To see that they are all necessary, note that the form of an invariant above
implies that for $m$ even, the Hilbert series of $(A^+)^H$ begins as:
$$1 + t^4 + t^8 + \cdots + t^{2m-4} + 2t^{2m} + t^{2m+2},$$
and for $m$ odd, it begins as
$$1 + t^4 + t^8 + \cdots + t^{2m-2} + t^{2m} + 2t^{2m+2}.$$
Therefore it is clear one needs the generators $u^4$ and $(uv)^m - (vu)^m$.
The subalgebra generated by these invariants is zero in degree $t^{2m+2}$
when $m$ is even and spanned by a power of $u^4$ in degree $t^{2m+2}$ when $m$ is odd.
In either case, $u^2((uv)^m + (vu)^m)$ is not generated by the other two
generators, hence all three are necessary.  Since $A$ has dimension 2, if the invariant ring were AS regular, it would also be AS regular of dimension 2, and hence by Lemma \ref{lem:notAS} the
invariant ring is not AS regular.
\end{proof}

%%%%%%%%%%%%%%%%%%%%%%%%%%
% Section 5: A_4m  m odd
%%%%%%%%%%%%%%%%%%%%%%%%%%

\section{\texorpdfstring{The Hopf algebras $\cA_{4m}$ of Masuoka for odd $m$}{The Hopf algebras A(4m) for odd m}}\label{sec:A4modd}

Recall that the Hopf algebras  $\cA_{4m}$ were defined in Section \ref{sec:B4m} as follows. The  group algebra
   $K = \kk[\langle a \rangle]$ of a cyclic group of order 2 is identified with its dual $\kk^{\langle a \rangle}$, in which $e_0$ and $e_1$ are the idempotents
$e_0 =(1+a)/2$ and $e_1 = (1-a)/2$. The Hopf algebras $\cA_{4m}$ are defined as algebras over $K$, with $K$ a central Hopf subalgebra (with $a$ group-like), and generated over $K$ by the two elements $s_+$ and $s_-$ with the relations :
$$s_{\pm}^2 = 1, \;\; (s_+s_-)^m = 1;$$
the coproduct, counit and antipode in $\cA_{4m}$ are:
$$\Delta(s_{\pm}) = s_{\pm}  \otimes e_0s_{\pm} + s_{\mp} \otimes e_1 s_\pm,\;\; \epsilon(s_\pm) = 1,\;\; S(s_\pm) = e_0s_\pm + e_1 s_\mp.$$

%%%%%%%%%%%%%%%%%%%%%%
% Subsection 5.1 The Grothendieck ring $K_0(\cA_{4m})$ ($m$ odd)
%%%%%%%%%%%%%%%%%%%%%%%%%%%

\subsection{\texorpdfstring{The Grothendieck ring $K_0(\cA_{4m})$ ($m$ odd)}{The Grothendieck ring K0(A(4m))}}

In this section $\lambda = e^{\frac{2 \pi \i}{m}}$ represents a primitive $m$th root of 
unity (not a $2m$th root of unity, as it was in the case of $\cB_{4m}$).
The next proposition is straightforward.
\begin{proposition} \label{prop:repA4m}
Let $m \in \NN$ be odd. The one-dimensional representations of $\cA_{4m}$  are of the form $T_{1,1,\pm1}$ and $T_{-1,-1,\pm1}$. The two-dimensional irreducible representations are
\[
\pi^\varepsilon_i(s_+)=\begin{pmatrix} 0&1\\1&0\end{pmatrix}, \quad \pi^\varepsilon_i(s_-)=\begin{pmatrix} 0& \lambda^{-i}\\\lambda^i&0\end{pmatrix}, \quad \pi^\varepsilon_i(a)=\begin{pmatrix}\varepsilon&0\\0&\varepsilon\end{pmatrix}
\]
where $\lambda =  e^{{2 \pi \i}/{m}}$ is a primitive $m$th root of unity, $i=1,\ldots,{(m-1)}/{2}$ and $\varepsilon=\pm1$. 
\end{proposition}

\begin{notation}\label{pi0A4m}
We denote by $\pi_0^\varepsilon$ the (reducible) two-dimensional representation given by
\[
\pi^\varepsilon_0(s_+)=\begin{pmatrix} 0&1\\1&0\end{pmatrix}, \quad  \pi^\varepsilon_0(s_-)=\begin{pmatrix} 0& 1\\ 1&0\end{pmatrix}, \quad
\pi^\varepsilon_0(a)=\begin{pmatrix}\varepsilon&0\\0&\varepsilon\end{pmatrix},
\]
with $\varepsilon=\pm1$.
A straightforward computation shows
\[
\pi^\varepsilon_0(u,v)=T_{1,1,\varepsilon}(u+v)\oplus T_{-1,-1,\varepsilon}(u-v).
\]
\end{notation}

\begin{theorem} \label{k04modd}
The ring $K_0(\cA_{4m})$ when $m$ is odd has the following fusion rules:
\[
\pi^\varepsilon_i(u,v)\otimes\pi^{+1}_j(x,y)=
\begin{cases}
\pi^{\varepsilon}_{i-j}(uy,vx)\oplus\pi^{\varepsilon}_{i+j}(ux,vy) & i+j\leq\frac{m-1}{2}, i-j\geq0\\
\pi^{\varepsilon}_{j-i}(vx,uy)\oplus\pi^{\varepsilon}_{i+j}(ux,vy) & i+j\leq\frac{m-1}{2}, i-j\leq0\\
\pi^{\varepsilon}_{i-j}(uy,vx)\oplus\pi^{\varepsilon}_{m-i-j}(vy,ux) & i+j>\frac{m-1}{2}, i-j\geq0\\
\pi^{\varepsilon}_{j-i}(vx,uy)\oplus\pi^{\varepsilon}_{m-i-j}(vy,ux) & i+j>\frac{m-1}{2}, i-j\leq0
\end{cases},
\]
\[
\pi^\varepsilon_i(u,v)\otimes\pi^{-1}_j(x,y)=
\begin{cases}
\pi^{-\varepsilon}_{i-j}(\lambda^ivy,ux)\oplus\pi^{-\varepsilon}_{i+j}(\lambda^ivx,uy) & i+j\leq\frac{m-1}{2}, i-j\geq0\\
\pi^{-\varepsilon}_{j-i}(ux,\lambda^ivy)\oplus\pi^{-\varepsilon}_{i+j}(\lambda^ivx,uy) & i+j\leq\frac{m-1}{2}, i-j\leq0\\
\pi^{-\varepsilon}_{i-j}(\lambda^ivy,ux)\oplus\pi^{-\varepsilon}_{m-i-j}(uy,\lambda^ivx) & i+j>\frac{m-1}{2}, i-j\geq0\\
\pi^{-\varepsilon}_{j-i}(ux,\lambda^ivy)\oplus\pi^{-\varepsilon}_{m-i-j}(uy,\lambda^ivx) & i+j>\frac{m-1}{2}, i-j\leq0
\end{cases},
\]
\[
T_{\pm1,\pm1,\varepsilon}(t)\otimes\pi^\delta_i(u,v)=\pi^{\varepsilon\delta}_i(tu,\pm tv),
\]
\[
\pi^\delta_i(u,v)\otimes T_{\pm1,\pm1,1}(t)=\pi^\delta_i(ut,\pm vt),
\]
\[
\pi^\delta_i(u,v)\otimes T_{\pm1,\pm1,-1}(t)=\pi^{-\delta}_i(\lambda^i vt,\pm ut),
\]
\[
T_{\pm1,\pm1,\varepsilon}(t)\otimes T_{1,1,\delta}(s)=T_{\pm1,\pm1,\varepsilon\delta}(ts),
\]
\[
T_{\pm1,\pm1,\varepsilon}(t)\otimes T_{-1,-1,\delta}(s)=T_{\mp1,\mp1,\varepsilon\delta}(ts).
\]
\end{theorem}
\begin{proof}
A computation shows that the $\cA_{4m}$ action on 
$\pi^\varepsilon_i(u,v)\otimes\pi^{+1}_j(x,y)$ is

\begin{center}
\bgroup
\def\arraystretch{1.25}
\begin{tabular}{|c||c|c|c|c|}
\hline
     & $ux$             & $uy$          & $vx$          & $vy$      \\ \hline \hline
$s_+$  & $vy$      & $vx$   & $uy$   & $ux$ \\ \hline
$s_-$  & $\lambda^{i+j}vy$      & $\lambda^{i-j}vx$   & $\lambda^{j-i}uy$   & $\lambda^{-(i+j)}ux$ \\ \hline
$a$  & $\varepsilon ux$ & $\varepsilon uy$ & $\varepsilon vx$ & $\varepsilon vy$  \\ \hline
\end{tabular}
\egroup
\end{center}
from which the first equality follows. The other equalities are proved similarly.
\end{proof}
\begin{proposition}\label{K0A4m=K0G}
When $m$ is odd, the ring $K_0(\cA_{4m})$ isomorphic to $K_0(D_{2m}\times\mathbb{Z}_2)$.
\end{proposition}
\begin{proof}
As algebras, $\cA_{4m}$ and $\kk [D_{2m}\times\mathbb{Z}_2]$ are isomorphic, hence 
$K_0(\cA_{4m})$ and $K_0(D_{2m}\times\mathbb{Z}_2)$ are isomorphic as abelian groups. 
Abusing notation, we denote the irreducible representations of
$\kk[D_{2m}\times\mathbb{Z}_2]$ in the same way we denoted the ones of $\cA_{4m}$. The 
multiplication $\pi^\varepsilon_i(u,v)\otimes\pi^{+1}_j(x,y)$ is clearly the same in both 
rings since in this case $s_+,s_-,a$ act on $\pi_i^\epsilon \otimes \pi_j^{+1}$
as group-likes for both algebras.  The table for the action of $D_{2m} \times \ZZ_2$
on $\pi^\varepsilon_i(u,v)\otimes\pi^{-1}_j(x,y)$ is
\begin{center}
\bgroup
\def\arraystretch{1.25}
\begin{tabular}{|c||c|c|c|c|}
\hline
     & $ux$             & $uy$          & $vx$          & $vy$      \\ \hline \hline
$s_+$  & $vy$      & $vx$   & $uy$   & $ux$ \\ \hline
$s_-$  & $\lambda^{i+j}vy$      & $\lambda^{i-j}vx$   & $\lambda^{j-i}uy$   & $\lambda^{-(i+j)}ux$ \\ \hline
$a$  & $-\varepsilon ux$ & $-\varepsilon uy$ & $-\varepsilon vx$ & $-\varepsilon vy$  \\ \hline
\end{tabular}
\egroup
\end{center}
which gives the following decomposition
\[
\pi^\varepsilon_i(u,v)\otimes\pi^{-1}_j(x,y)=
\begin{cases}
\pi^{-\varepsilon}_{i-j}(uy,vx)\oplus\pi^{-\varepsilon}_{i+j}(ux,vy) & i+j\leq\frac{m-1}{2}, i-j\geq0\\
\pi^{-\varepsilon}_{j-i}(uy,vx)\oplus\pi^{-\varepsilon}_{i+j}(ux,vy) & i+j\leq\frac{m-1}{2}, i-j\leq0\\
\pi^{-\varepsilon}_{i-j}(vx,uy)\oplus\pi^{-\varepsilon}_{m-i-j}(ux,vy) & i+j>\frac{m-1}{2}, i-j\geq0\\
\pi^{-\varepsilon}_{j-i}(vx,uy)\oplus\pi^{-\varepsilon}_{m-i-j}(ux,vy) & i+j>\frac{m-1}{2}, i-j\leq0
\end{cases},
\]
hence the product is the same. The other products are similarly checked.
\end{proof}
\begin{theorem} \label{A4mOddIF}
Let $A$ be a graded $\kk$-algebra generated in degree 1 by $u$ and $v$ with $A_1 = \kk u\oplus \kk v=\pi^\delta_i(u,v)$. The action of $H= \cA_{4m}$, $m$ odd, on $A$ is inner-faithful if and only if $\delta=-1$ and $(i,m)=1$ { for $i = 1, \ldots, (m-1)/2$}.
\end{theorem}

\begin{proof}
If $\delta=+1$ then the Hopf ideal generated by $1-a$ annihilates $A$ and hence the
action is not inner-faithful by Lemma \ref{HopfIdeal}. From now on we will assume 
$\delta=-1$.  By \ref{K0A4m=K0G} a representation $\pi^{-1}_i$ generates $K_0(\cA_{4m})$ 
if and only if it generates $K_0(D_{2m}\times\mathbb{Z}_2)$, so from now on we will work 
with the group algebra $\kk G$ with $G=D_{2m}\times\mathbb{Z}_2$, and $A$ will be a
graded $\kk$-algebra generated in degree 1 by $u$ and $v$ with
$A_1 = \kk u\oplus \kk v=\pi^{-1}_i(u,v)$ on which $G$ acts.

Let $I$ be a Hopf ideal of $\kk G$ such that $I A=0$ with $I \neq 0$.  Then by
\cite[Lemma 1.4]{P} there is a normal subgroup $N$ of $G$ such that
$I=(\kk G)(\kk N)^+$.  Since $I$ is not trivial then neither is $N$.  Furthermore,
since $a$ acts as a scalar matrix, there must be an integer $p$ with $0 \leq p \leq m-1$
such that either $s_-(s_+s_-)^p\in N$ or $(s_+s_-)^p\in N$ with $p \neq 0$.
We may now conclude as in the proof of Theorem \ref{thm:B4mIF} that $(i,m) \neq 1$.

Hence, if the action is not inner-faithful then $(i,m)\neq 1$. If $(i,m)\neq 1$, then 
choose $p$ between 1 and $m-1$ such that $ip$ is a multiple of $m$, so that the Hopf
ideal generated by $1-(s_+s_-)^p$ annihilates $A$, and hence the action is not 
inner-faithful.
\end{proof}

%%%%%%%%%%%%%%%%%%%%%
% Subsectionn 5.2: Inner-faithful Hopf actions of \texorpdfstring{$\cA_{4m}$ ($m$ odd)
%%%%%%%%%%%%%%%%%%%%%%%

\subsection{Inner-faithful Hopf actions of \texorpdfstring{$\cA_{4m}$ ($m$ odd) on AS regular algebras and their fixed rings}{Hopf actions of A(4m)}}
Recall that we are assuming that 
 $m$ is odd, and that  $\lambda = e^{{2 \pi \i}/{m}}$ is a primitive $m$th root of unity.
The Hopf algebras $H= \cA_{4m}$ act on the AS regular algebras of dimension 2
\[
A^\pm=\frac{\kk\langle u,v\rangle}{(u^2\pm\lambda^iv^2)},\quad\mathrm{where }\;\; \kk u\oplus \kk v=\pi^{-1}_i(u,v).
\]
We denote  the algebra $A^-$ by $A$ and the Hopf algebra ${\cA_{4m}}$ by $H$, and we compute the fixed ring $A^{H}$.

\begin{theorem} \label{thm:A4mOddFR}
If $H = \cA_{4m}$, $m$ odd, acts on 
$$A = A^-= \frac{\kk\langle u,v\rangle}{(u^2-\lambda^iv^2)}$$ for $\lambda = e^{{2 \pi \i}/{m}}$ by
$\pi^{-1}_i(u,v)$ 
inner-faithfully (i.e., $(i,m)=1$ { for $i = 1,\ldots,(m-1)/2$}), then
\[
A^{H}=\kk[u^2,(uv)^m+(vu)^m].
\]
Hence  $\cA_{4m}$ is a reflection Hopf algebra for { this action on } $A^-$ when $(i,m)=1$.
\end{theorem}
\begin{proof} 
A straightforward calculation shows that $u^2$ and $(uv)^m + (vu)^m$ are invariant.
More generally, if an element is fixed by $a$ then every monomial in it must have
even total degree.
By Lemma \ref{lem:u2v2Basis}, any homogeneous element in $A$ of even degree must be
of the form
\[
F=\sum_{p,q}\alpha_{p,q}u^{2p}(vu)^q+\sum_{p,q}\beta_{p,q}u^{2p}(uv)^q.
\]

The action of $s_+$ and $s_-$ on the basis used in the expression above is:
\begin{align*}
s_+(u^{2p}(vu)^{q})&=\lambda^{-iq}u^{2p}(uv)^q\\
s_+(u^{2p}(uv)^{q})&=\lambda^{iq}u^{2p}(vu)^q\\
s_-(u^{2p}(vu)^{q})&=\lambda^{iq}u^{2p}(uv)^q\\
s_-(u^{2p}(uv)^{q})&=\lambda^{-iq}u^{2p}(vu)^q.
\end{align*}
Setting $s_+(F)=F$ gives $\alpha_{p,q}=\beta_{p,q}\lambda^{iq}$, and 
setting $s_-(F)=F$ gives $\alpha_{p,q}=\beta_{p,q}\lambda^{-iq}$
for all $p$.  It follows that $\beta_{p,q}\neq0$ implies $\lambda^{2iq}=1$.
Since $m$ is odd and $(i,m)=1$, we have $q\equiv0\pmod{m}$.  Using
the fact that $\lambda$ is an $m^\text{th}$ root of unity, we have that
an invariant $F$ has the form
\[
F=\sum_{p,q}\alpha_{p,q}u^{2p}[(vu)^{qm}+(uv)^{qm}].
\]

By Lemma \ref{binomial} with $X= (uv)^m$ and $Y= (vu)^m$, it follows that 
$(vu)^{km}+(uv)^{km}$ can be generated by $u^2$ and $(vu)^m+(uv)^m$.
That the generators of the fixed ring are algebraically independent follows
using the same argument as in Theorem \ref{B4mfixed}.
\end{proof}

\begin{theorem}
If $H = \cA_{4m}$, $m$ odd, acts on 
$$ A^+ = \frac{\kk\langle u,v\rangle}{(u^2+\lambda^iv^2)}$$ for $\lambda= e^{{2 \pi \i}/{m}}$ by
$\pi^{-1}_i(u,v)$,
inner-faithfully (i.e., $(i,m)=1$ { for $i = 1, \ldots, (m-1)/2$}), the
fixed subring is
\[
(A^+)^H=\kk[u^4,(vu)^m+(uv)^m,u^2((vu)^m-(uv)^m)].
\]
Furthermore, the ring $(A^+)^H$ is not AS regular.
\end{theorem}
\begin{proof}
One can check that an invariant must be of the form
\[
\sum_{p=0}^{d-1}\alpha_{2p}u^{2p}((vu)^{km}+(-1)^{p+k}(-(uv)^m)^k)
\]
with $d-p=km$. Now we conclude using Remark \ref{GenInvNow}, by setting
$z=u^2,x=(vu)^m,y=-(uv)^m$.  The proof that these generators are necessary is then
the same as that of the case $m$ odd in Theorem \ref{thm:B4mPlus}.
\end{proof}

%%%%%%%%%%%%%%%%%%%%%%%%%%%%%%%%%%%%%%
%%   Section 6;  A_4m  m even
%%%%%%%%%%%%%%%%%%%%%%%%%%%%%%%%%%%

\section{\texorpdfstring{The Hopf algebras $\cA_{4m}$ of Masuoka for even $m$}{The Hopf algebras A(4m) for even m}} \label{sec:A4meven}

%%%%%%%%%%%%%%%%%%
%  Subsection 6.1 The Grothendieck ring $K_0(\cA_{4m})$ ($m$ even)
%%%%%%%%%%%%%%%%%%%%

\subsection{\texorpdfstring{The Grothendieck ring $K_0(\cA_{4m})$ ($m$ even)}{The Grothendieck ring K0(A(4m))}}

The next proposition is straightforward.
\begin{proposition} \label{prop:repA4mEven}
If $m$ is even then the one-dimensional irreducible representations of $\cA_{4m}$ are of the form $T_{1,1,\pm1}$, $T_{-1,-1,\pm1}$, $T_{1,-1,\pm1}$ and $T_{-1,1,\pm1}$. The two-dimensional irreducible representations are 
\[
\pi^\varepsilon_i(s_+)=\begin{pmatrix} 0&1\\1&0\end{pmatrix}, \pi^\varepsilon_i(s_-)=\begin{pmatrix} 0& \lambda^{-i}\\\lambda^i&0\end{pmatrix},\pi^\varepsilon_i(a)=\begin{pmatrix}\varepsilon&0\\0&\varepsilon\end{pmatrix}
\]
where $\lambda=e^{{2 \pi \i}/{m}}$ is a primitive $m$th root of unity, $i=1,\ldots,{m}/{2}-1$ and $\varepsilon=\pm 1$.

\end{proposition}
\begin{notation}
Similar to Notation \ref{pi0A4m}, we define 
the following (reducible) two-dimensional 
representations below, where
$\epsilon = \pm 1$:
\begin{align*}
\pi^\varepsilon_0(s_+)&=\begin{pmatrix} 0&1\\1&0\end{pmatrix},&
\pi^\varepsilon_0(s_-)&=\begin{pmatrix} 0& 1\\ 1&0\end{pmatrix},&
\pi^\varepsilon_0(a)&=\begin{pmatrix}\varepsilon&0\\0&\varepsilon\end{pmatrix},\\
\pi^\varepsilon_{{m}/{2}}(s_+)&=\begin{pmatrix} 0&1\\1&0\end{pmatrix},&
\pi^\varepsilon_{{m}/{2}}(s_-)&=\begin{pmatrix} 0&-1\\-1&0\end{pmatrix},&
\pi^\varepsilon_{{m}/{2}}(a)&=\begin{pmatrix} \varepsilon&0\\0&\varepsilon\end{pmatrix},\end{align*}
A straightforward computation shows that:
\begin{align*}
\pi^\varepsilon_0(u,v)&=T_{1,1,\varepsilon}(u+v)\oplus T_{-1,-1,\varepsilon}(u-v), \\
\pi^\varepsilon_{{m}/{2}}(u,v)&=T_{1,-1,\varepsilon}(u+v)\oplus T_{-1,1,\varepsilon}(u-v).
\end{align*}
\end{notation}

The proof of the following theorem is similar
to that of Theorem \ref{k04modd}.
\begin{theorem}\label{K0A4mEven}
The ring $K_0(\cA_{4m})$ when $m$ is even has the following fusion rules:
\[
\pi^\varepsilon_i(u,v)\otimes\pi^{+1}_j(x,y)=
\begin{cases}
\pi^{\varepsilon}_{i-j}(uy,vx)\oplus\pi^{\varepsilon}_{i+j}(ux,vy) & i+j\leq {m}/{2}-1, i-j\geq0\\
\pi^{\varepsilon}_{j-i}(vx,uy)\oplus\pi^{\varepsilon}_{i+j}(ux,vy) & i+j\leq {m}/{2}-1, i-j\leq0\\
\pi^{\varepsilon}_{i-j}(uy,vx)\oplus\pi^{\varepsilon}_{m-i-j}(vy,ux) & i+j>{m}/{2}-1, i-j\geq0\\
\pi^{\varepsilon}_{j-i}(vx,uy)\oplus\pi^{\varepsilon}_{m-i-j}(vy,ux) & i+j>{m}/{2}-1, i-j\leq0
\end{cases},
\]
\[
\pi^\varepsilon_i(u,v)\otimes\pi^{-1}_j(x,y)=
\begin{cases}
\pi^{-\varepsilon}_{i-j}(\lambda^ivy,ux)\oplus\pi^{-\varepsilon}_{i+j}(\lambda^ivx,uy) & i+j\leq {m}/{2}-1, i-j\geq0\\
\pi^{-\varepsilon}_{j-i}(ux,\lambda^ivy)\oplus\pi^{-\varepsilon}_{i+j}(\lambda^ivx,uy) & i+j\leq {m}/{2}-1, i-j\leq0\\
\pi^{-\varepsilon}_{i-j}(\lambda^ivy,ux)\oplus\pi^{-\varepsilon}_{m-i-j}(uy,\lambda^ivx) & i+j>{m}/{2}-1, i-j\geq0\\
\pi^{-\varepsilon}_{j-i}(ux,\lambda^ivy)\oplus\pi^{-\varepsilon}_{m-i-j}(uy,\lambda^ivx) & i+j>{m}/{2}-1, i-j\leq0
\end{cases},
\]\belowdisplayskip=-12pt
\begin{align*}
T_{\pm1,\pm1,\varepsilon}(t)\otimes\pi^\delta_i(u,v)   & = 
\pi^{\varepsilon\delta}_i(tu,\pm tv), \\
\pi^\delta_i(u,v)\otimes T_{\pm1,\pm1,1}(t)            & =  \pi^\delta_i(ut,\pm vt), \\
\pi^\delta_i(u,v)\otimes T_{\pm1,\pm1,-1}(t)           & =  \pi^{-\delta}_i(\lambda^i vt,\pm ut) \\
T_{\pm1,\mp1,\varepsilon}(t)\otimes\pi^\delta_{i}(u,v) & =  \pi^{\varepsilon\delta}_{{m}/{2}-i}(\pm\delta tv,tu) \\
\pi^\delta_i(u,v)\otimes T_{\pm1,\mp1,1}(t)            & =  \pi^\delta_{{m}/{2}-i}(\pm vt,ut) \\
\pi^\delta_i(u,v)\otimes T_{\pm1,\mp1,-1}(t)           & =  \pi^{-\delta}_{\frac{m}{2}-i}(\pm ut,\lambda^ivt) \\
T_{\alpha,\alpha,\epsilon}(t) \otimes T_{\beta,\beta,\delta}(s) & =  T_{\alpha\beta,\alpha\beta,\epsilon\delta}(ts) \\
T_{\alpha,-\alpha,\epsilon}(t) \otimes T_{\beta,-\beta,\delta}(s) & = 
T_{\alpha\beta\delta,\alpha\beta\delta,\epsilon\delta}(ts) \\
T_{\alpha,\alpha,\epsilon}(t) \otimes T_{\beta,-\beta,\delta}(s) & =  
T_{\alpha\beta,-\alpha\beta,\epsilon\delta}(ts)\\
T_{\alpha,-\alpha,\epsilon}(t) \otimes T_{\beta,\beta,\delta}(s) & = 
T_{\alpha\beta\delta,-\alpha\beta\delta,\epsilon\delta}(ts) 
\end{align*}
\qed
\end{theorem}

\begin{remark}
In order to make the decompositions above more palatable, we introduce
a final piece of notation.  For $\frac{m}{2} < j < m$, we set $\pi_j^\epsilon = \pi_{m-j}^\epsilon$.
Then the tensor product decompositions above, given without reference to a basis, become the following
for all $j,k \in \{0,\dots,m-1\}$, where
the subscripts are read modulo $m$:
\begin{eqnarray}
\pi_j^\epsilon \otimes \pi_k^\delta & \cong & \pi_{|j-k|}^{\epsilon\delta} \oplus \pi_{j+k}^{\epsilon\delta} \nonumber \\
T_{\pm 1,\pm 1, \epsilon} \otimes \pi_j^\delta & \cong & \pi_j^{\epsilon\delta} =
  \pi_j^\delta \otimes T_{\pm 1,\pm 1,\epsilon} \nonumber \\
T_{\pm 1,\mp 1, \epsilon} \otimes \pi_j^\delta & \cong & \pi_{{m}/{2} - j}^{\epsilon\delta} \cong
  \pi_j^\delta \otimes T_{\pm 1,\mp 1,\epsilon} \nonumber \\
\label{K0A4mEvenEq}
T_{\alpha,\alpha,\epsilon} \otimes T_{\beta,\beta,\delta} & \cong & T_{\alpha\beta,\alpha\beta,\epsilon\delta} \\
T_{\alpha,-\alpha,\epsilon} \otimes T_{\beta,-\beta,\delta} & \cong & 
   T_{\alpha\beta\delta,\alpha\beta\delta,\epsilon\delta} \nonumber \\
T_{\alpha,\alpha,\epsilon} \otimes T_{\beta,-\beta,\delta} & \cong & 
  T_{\alpha\beta,-\alpha\beta,\epsilon\delta} \nonumber \\
T_{\alpha,-\alpha,\epsilon} \otimes T_{\beta,\beta,\delta} & \cong & 
   T_{\alpha\beta\delta,-\alpha\beta\delta,\epsilon\delta} \nonumber
\end{eqnarray}
Note that the group of one-dimensional representations is isomorphic 
to the dihedral group of order 8, where the quarter rotations are given by 
$T_{1,-1,-1}$ and $T_{-1,1,-1}$,
the center is generated by $T_{-1,-1,1}$, and the reflections are 
$T_{1,1,-1},T_{-1,1,1},T_{1,-1,1}$ and $T_{-1,-1,-1}$. Since the 
Grothendieck ring is not commutative, it cannot be isomorphic to the 
Grothendieck ring of a group, and therefore an approach different from  
those used in the earlier examples is needed.
\end{remark}

\begin{theorem} \label{no2dimlgenerates}
When $m$ is even, none of the representations $\pi^\varepsilon_i$ of $\cA_{4m}$ for $i = 1, \ldots, m/2-1, \epsilon = \pm1$  is inner-faithful.
\end{theorem}
\begin{proof}
The two-dimensional representations with positive exponent cannot be inner-faithful 
because they cannot generate a two-dimensional irreducible representation with 
a negative exponent.  A necessary condition for $\pi^{-1}_i$ to be inner-faithful is 
$(i,m)=1$, since all the two-dimensional irreducible representations appearing as 
direct summands of a tensor power of $\pi^{-1}_i$ have an index that is an integer 
combination of $i$ and $m$; if $\pi^{-1}_i$ were inner-faithful then one of these 
combinations would be equal to 1, showing that $(i,m)=1$. Hence we assume that 
$(i,m)=1$.

By the first equality in \eqref{K0A4mEven}, the two-dimensional summands of 
$(\pi_i^{-1})^{\otimes \ell}$ are of the form $\pi_{ij}^{(-1)^j}$ for
$j = 1,\dots,\ell$.  Therefore, we obtain (at most) one-half of the two-dimensional 
irreducible representations of $\cA_{4m}$ from $\pi_{ij}^{(-1)^j}$.
\end{proof}

Before characterizing the three-dimensional inner-faithful representations, we prove
a lemma which aids us in our reductions.

\begin{lemma}\label{LemmaA4mEvenIF}
When $m$ is even, a representation $V$ of $\cA_{4m}$ is inner-faithful if
and only if it generates all the irreducible two-dimensional 
representations up to a sign and $\pi_i^{+1}$ and $\pi_i^{-1}$ for 
some $i$.
\end{lemma}

\begin{proof}
One implication is obvious.   The fusion rules show that by 
decomposing $(\pi_i^{+1})^2$, $\pi_i^{+1}\otimes\pi_i^{-1}$, 
$\pi^\varepsilon_{{m}/{2}-i}\otimes\pi^{\varepsilon}_i$ and
$\pi^{\varepsilon}_{{m}/{2}-i}\otimes\pi^{-\varepsilon}_i$, $V$ generates all the one-dimensional irreducible
representations.  By tensoring an irreducible two-dimensional 
representation by $T_{1,1,-1}$ we change its exponent.  It follows
that $V$ is inner-faithful.
\end{proof}

\begin{theorem}\label{A4mEvenIF}
When $m$ is even, the inner-faithful three-dimensional representations of $\cA_{4m}$
are given as follows, where $(i,m) = 1$ { and  $i = 1, \ldots, m/2-1$}:
\begin{enumerate}
\item $\pi_i^{+1}\oplus T_{\pm1,\pm1,-1}$,

\item $\pi_i^{-1}\oplus T_{\pm1,\pm1,-1}$,

\item $\pi^{+1}_i\oplus T_{\pm1,\mp1,-1}$,

\item $\pi^{-1}_i\oplus T_{\pm1,\mp1,-1}$, provided
$m\equiv0\pmod{4}$

\item $\pi^{-1}_i\oplus T_{\pm1,\mp1,1}$, provided
$m\equiv 2\pmod{4}$.

\end{enumerate}
\end{theorem}

\begin{proof}
In this proof, we make frequent use of the relations appearing in
equations { \eqref{K0A4mEvenEq}}. Let $V$ be an inner-faithful three-dimensional representation
of $\cA_{4m}$. Certainly $V$ must have a two-dimensional irreducible summand, so we let
$V = \pi_i^\epsilon \oplus T_{\alpha,\beta,\gamma}$.

If $\alpha = \beta$ and $\gamma = 1$, then the second isomorphism in equations \eqref{K0A4mEvenEq} shows that
$T_{\alpha,\beta,\gamma}$ does not assist $\pi_i^\epsilon$ in generating new irreducible
two-dimensional representations, so Theorem \ref{no2dimlgenerates} gives a contradiction.
If $\alpha = \beta$ and $\gamma = -1$, then one must have that $(i,m) = 1$
in order to obtain an irreducible two-dimensional representation with subscript
1, since tensoring with $T_{\alpha,\alpha,-1}$ does not change the subscript.
Cases (1) and (2) then follow by Lemma \ref{LemmaA4mEvenIF}.

Therefore we may assume that $\beta = -\alpha$.   For $V$ to be inner-faithful,
the first and third isomorphisms in equations \eqref{K0A4mEvenEq} show that $(i,{m}/{2}) = 1$.
Hence either $(i,m) = 1$, or $i$ is even and $m \equiv 2\;(4)$.
In the latter case, $\pi_i^\epsilon$ only generates two-dimensional representations
of the form $\pi^{\epsilon^s}_{si}$, which all have even subscripts.  Since ${m}/{2}$ is odd,
tensoring with $T_{\alpha,-\alpha,\gamma}$ only generates representations of the form
$\pi^{\epsilon^s\gamma}_t$ where $t={m}/{2}-si$ is odd, and hence cannot yield a pair of irreducible
two-dimensional representations with the same subscript but different exponent.
Therefore, we have that $(i,m) = 1$, and hence we are in a situation where all two-dimensional
representations are generated up to sign.

If $\epsilon = 1$, then one must have $\gamma = -1$, otherwise $V$ does not
generate any `negative' representations.  The case $\gamma = -1$ indeed gives an
inner-faithful representation since $\pi_i^{+1}$ generates all `positive'
two-dimensional representations and tensoring with $T_{\pm 1,\mp 1,-1}$
gives one of the `negative' two-dimensional representations required by Lemma \ref{LemmaA4mEvenIF}.

Next, suppose $\epsilon = -1$. Then $\pi_i^{-1}$ generates $\pi_{si}^{(-1)^s}$.
Since $i$ is odd, if $si$ is odd, then $s$ is also odd,
and if $si$ is even, then $s$ is also even.  Therefore all irreducible two-dimensional
representations generated by $\pi_i^{-1}$ with even (respectively, odd) subscripts have
positive (respectively, negative) exponents.  If
$m \equiv 2\pmod{4}$, then ${m}/{2}$ is
odd, so that while $T_{\alpha,-\alpha,-1}$ does not help generate any new two-dimensional representations,
$T_{\alpha,-\alpha,1}$ does.  This gives us case (5) in the table above.
Similarly when $m \equiv 0\pmod{4}$, ${m}/{2}$ is even, the
roles of $T_{\alpha,-\alpha,-1}$ and $T_{\alpha,-\alpha,1}$ are reversed, giving us case (4).
\end{proof}

%%%%%%%%%%%%%%%%%%%%%%%%%%%%%%%
%  Subsection 6.2 Inner-faithful Hopf actions of {$\cA_{4m}$ ($m$ even) on AS regular algebras and their fixed rings
%%%%%%%%%%%%%%%%%%%%%%%%%%%%%%%

\subsection{Inner-faithful Hopf actions of \texorpdfstring{$\cA_{4m}$ ($m$ even) on AS regular algebras and their fixed rings}{Hopf actions of A(4m)}}

When $m$ is even then, by {Theorem} \ref{K0A4mEven}, $H=\cA_{4m}$ acts on the following AS regular Ore extensions of dimension 3, where $\varepsilon=\pm1$ { and $i = 1, \ldots, m/2-1$}; recall that $\lambda= e^{{2 \pi \i}/{m}}$ is a primitive $m$th root of unity .
\begin{align*}
&A_{1,\varepsilon}^{\pm}=\frac{\kk \langle u,v\rangle}{(uv\pm vu)}[t;\sigma],&\quad\sigma=\begin{pmatrix} 0&1\\\lambda^i&0\end{pmatrix},&\quad \kk u\oplus \kk v=\pi^{+1}_i,&\quad \kk t=T_{\varepsilon,\varepsilon,-1},\\
&A_{2,\varepsilon}^{\pm}=\frac{\kk \langle u,v\rangle}{(u^2\pm \lambda^iv^2)}[t;\sigma],&\quad\sigma=\begin{pmatrix} 0&1\\\lambda^i&0\end{pmatrix},&\quad \kk u\oplus \kk v=\pi^{-1}_i,&\quad \kk t=T_{\varepsilon,\varepsilon,-1},\\
&A_{3,\varepsilon}^{\pm}=\frac{\kk \langle u,v\rangle}{(uv\pm vu)}[t;\sigma],&\quad\sigma=\begin{pmatrix} 0&1\\\lambda^i&0\end{pmatrix},&\quad \kk u\oplus \kk v=\pi^{+1}_i,&\quad \kk t=T_{\varepsilon,-\varepsilon,-1},\\
\end{align*}
\begin{align*}
&A_{4,\varepsilon}^{\pm}=\frac{\kk\langle u,v\rangle}{(u^2\pm \lambda^iv^2)}[t;\sigma],&\quad\sigma=\begin{pmatrix} 0&-1\\\lambda^i&0\end{pmatrix},&\quad \kk u\oplus \kk v=\pi^{-1}_i,&\quad \kk t=T_{\varepsilon,-\varepsilon,-1},\\
&A_{5,\varepsilon}^{\pm}=\frac{\kk \langle u,v\rangle}{(u^2\pm \lambda^iv^2)}[t;\sigma],&\quad\sigma={ \begin{pmatrix} 1 &0\\0 & -1\end{pmatrix}},&\quad \kk u\oplus \kk v=\pi^{-1}_i,&\quad \kk t=T_{\varepsilon,-\varepsilon,1}.
\end{align*}
Note that in each case $\sigma$ is an automorphism of the coefficient ring.

Before we analyze these actions, we prove a lemma which helps shrink the search
space for invariants.

\begin{lemma} \label{lem:oreInvariants}
Let $H$ be a finite dimensional semisimple Hopf algebra and let $A$ be an 
algebra which is a domain.  Suppose that $H$ acts linearly on the Ore 
extension $A[t;\sigma]$ such that the $H$ action on $A[t;\sigma]$
restricts to an $H$-action on $A$, and that $T = \kk t$ is a 
one-dimensional $H$-module.
If $T^{\otimes k}$ is not a direct summand of $A$ as an $H$-module, then 
there are not any nonzero elements in $A[t;\sigma]^H$ of the form $ft^k$, 
where $f \in A$.
\end{lemma}

\begin{proof}
Let $ft^k$ be an $H$-invariant.  Let $W = Hf$ be the $H$-submodule
of $A$ generated by $f$.  Then the $H$-submodule $Wt^k$ of $A[t;\sigma]$ 
is isomorphic to $W \otimes T^{\otimes k}$.  Since 
$ft^k$ is an invariant, the trivial representation must
appear as a summand of $W \otimes T^{\otimes k}$.  It follows that the 
inverse $U$ (in $K_0(H)$) of $T^{\otimes k}$ must appear as a direct 
summand of $W$.  Since $U$ has finite order in $K_0(H)$
and $A$ is a domain, we must therefore have that
$T^{\otimes k}$ appears as an $H$-direct summand of $A$.
\end{proof}

\begin{proposition} \label{prop:noOddDegree}
For any of the actions of $\cA_{4m}$ ($m$ even) on the Ore
extensions given above, all invariants involve only even powers of $t$.
\end{proposition}

\begin{proof}
In all cases, the hypotheses of Lemma \ref{lem:oreInvariants} are met for 
all odd $k$.
\end{proof}

\begin{theorem}
\label{A4mfixedringsEven}
The fixed rings for the inner-faithful actions of $H=\cA_{4m}$, $m$ even (see Theorem \ref{A4mEvenIF}), on the previous algebras $A^-_{k, \varepsilon}$ are:
\begin{align*}
(A^-_{1,\varepsilon})^H&=\kk[uv,u^m+v^m][t^2; \sigma'],\\ 
(A^-_{2,\varepsilon})^{H}&=\kk[u^2,(vu)^{{m}/{2}}-(uv)^{{m}/{2}}][t^2;\sigma'],\\ 
(A^-_{3,\varepsilon})^{H}&=\kk\langle uv, u^m+v^m, (u^m-v^m)t^2,  t^4\rangle,\\
(A^-_{4,\varepsilon})^{H} & =\kk\langle u^2, (vu)^{{m}/{2}}-(uv)^{{m}/{2}}, ((vu)^{{m}/{2}}+(uv)^{{m}/{2}})t^2, t^4 \rangle,\\
(A^-_{5,\varepsilon})^{H} &=\kk[u^2,(vu)^{{m}/{2}}-(uv)^{{m}/{2}}][t^2;\sigma'],
\end{align*}
where $\sigma'$ is the induced automorphism.
Furthermore, $H= \cA_{4m}$ ($m$ even) is a reflection Hopf algebra for
$A = A^-_{1,\varepsilon}$, $A^-_{2,\varepsilon}$, $A^-_{5,\varepsilon}$, but not for 
$A^-_{3,\varepsilon}$ and $A^-_{4,\varepsilon}$.
\end{theorem}

\begin{proof}
By Proposition \ref{prop:noOddDegree}, we may assume all invariants 
in all cases involve only even powers of $t$.  Also, since the 
actions in each case are graded we may assume all invariants
are homogeneous.

Now consider an invariant $F \in A_{1,\epsilon}^-$, written in the 
usual basis of $A^-_{1,\epsilon}$:
\[
F = \sum_{p,q,r}\alpha_{p,q,r}u^pv^qt^{2r}.
\]
A computation shows that
\begin{align*}
s_+(u^pv^qt^{2r})&=u^qv^pt^{2r}, &
s_-(u^pv^qt^{2r})&=\lambda^{i(p-q)}u^qv^pt^{2r}.
\end{align*}
Setting $s_+F=F$ and $s_-F=F$ yields $i(p-q)\equiv0\pmod{m}$ and 
$\alpha_{p,q,r}=\alpha_{q,p,r}$. This implies that
$p\equiv q\pmod{m}$.  After relabeling the coefficients, we have that
$F$ must therefore be of the form
\[
F=\sum_{q,k,r}\alpha_{q,k,r}u^qv^q(u^{km}+v^{km})t^{2r},
\]
Lemma \ref{binomial} shows that $F$ can be generated by the
invariants
$uv,u^{m}+v^{m},t^2$. The fixed ring is AS regular since 
$uv,u^{m}+v^{m}$ generate a commutative polynomial ring, and 
$\Bbbk[uv,u^{m}+v^{m}]\langle t^2 \rangle$ is an Ore extension, 
where $\sigma'$ is the induced automorphism, and hence the fixed 
ring is AS regular.

Since $A^-_{2,\epsilon}$ is an Ore extension of an algebra of the 
form appearing in Lemma \ref{lem:u2v2Basis}, we will use the basis
$\{u^p(vu)^qv^\delta t^r\}$ (where $p,q,r$ are nonnegative integers
and $\delta$ is $0$ or $1$).  Since the power of $t$ must be even,
the action of $a$ on the base of the Ore extension shows that
the total degree of an invariant must be even as well.

Since $t^2$ is fixed, it is enough to find the elements in the 
subalgebra generated by $u$ and $v$ that are fixed by $s_+$ and 
$s_-$.  A computation shows that the action of $s_\pm$
on the following monomials, where $p$ is odd and $q$ is arbitrary, 
is:
\begin{equation}\label{ActionOnpi-1i2}
\begin{split}
s_+(u^p(vu)^qv)        &= \lambda^{i(q+1)}u^{p-1}(vu)^{q+1} \\
s_+(u^{p-1}(vu)^{q+1}) &= \lambda^{-i(q+1)}u^p(vu)^qv \\
s_-(u^p(vu)^qv)        &= \lambda^{-i(q+1)}u^{p-1}(vu)^{q+1} \\
s_-(u^{p-1}(vu)^{q+1}) &= \lambda^{i(q+1)}u^p(vu)^qv.
\end{split}
\end{equation}
Now suppose $F$ is an invariant in $u$ and $v$ alone:
\[
F=\sum_{p\;\mathrm{odd},q}\alpha_{p,q}u^p(vu)^qv+\sum_{p\;\mathrm{odd},q}\beta_{p,q}u^{p-1}(vu)^{q+1}.
\]
Using \eqref{ActionOnpi-1i2} we see that in order for $F$ to be 
invariant $q+1$ must be congruent to zero modulo ${m}/{2}$.  After 
reindexing the coefficients, $F$ may be written as:
\[
F=\sum_{k,\ell}\alpha_{k,\ell} u^{2k}\left((uv)^{{m}\ell/2}+(-1)^\ell (vu)^{{m}\ell/2}\right).
\]
Hence Lemma \ref{binomial} shows that all invariants in $k\langle u, v \rangle$ are generated by $u^2$ and $(uv)^{{m}/{2}}-(vu)^{{m}/{2}}$. The fixed ring is AS regular since $u^2, (uv)^{m/{2}}-(vu)^{{m}/{2}}$ generate a commutative polynomial ring, and $\Bbbk[u^2,(uv)^{{m}/{2}}-(vu)^{{m}/{2}}][ t^2; \sigma']$ is an Ore extension, where $\sigma'$ is the induced automorphism.

For $A^-_{3,\varepsilon}$ one can show by an argument similar to that for $A_{1,\epsilon}^-$
that an invariant must have the form
\[
\sum_{k,q,r}\alpha_{k,q,r}(uv)^q((u^k)^m+(-1)^{r+k}(-v^m)^k)t^{2r}.
\]
Using Remark \ref{GenInvNoz} with $z=uv,x=u^m,y=-v^m,w=t^2$,
we see that $uv,u^m+v^m,(u^m-v^m)t^2$ and $t^4$ generate the invariant subring.
To see the proposed generators are all necessary, note that the algebra $A^-_{3,\epsilon}$
is bigraded by setting the bidegree of $u$ and $v$ to be $(1,0)$ and the bidegree
of $t$ to be $(0,1)$.  By the above description of the invariants, the bigraded Hilbert 
series of $(A_{3,\epsilon}^-)^H$ in bidegree less than $(m,4)$ in lexicographic order is
given by
$$1 + s_1^2 + \cdots + s_1^{m-2} + 2s_1^m + s_1^ms_2^2 +
  (1 + s_1^2 + \cdots + 2s_1^m)s_2^4,$$
where we use $s_1$ to represent bidegree $(1,0)$ and $s_2$ to represent $(0,1)$.
Therefore the subalgebra generated by $uv$ and $t^4$ is spanned by $(uv)^\frac{m}{2}$
in bidegree $(m,0)$ and is zero in bidegree $(m,2)$, hence the generators $u^m + v^m$ 
and $(u^m - v^m)t^2$ are both necessary.  By Lemma \ref{lem:notAS}
the invariant ring is not AS regular.

For $A_{4,\varepsilon}^-$, again using Proposition 
\ref{prop:noOddDegree} and the definition of the action one may show
that an invariant must be of the form
\[
\sum_{k,\ell,r}\alpha_{k,\ell,r}u^{2k}\left(((uv)^{{m}/{2}})^\ell+(-1)^{\ell+r}((vu)^{{m}/{2}})^\ell \right)t^{2r}.
\]
Now we conclude that $u^2,(vu)^{{m}/{2}} - (uv)^{{m}/{2}},((vu)^{{m}/{2}} + (uv)^{m/{2}})t^2$
and $t^4$ generate the invariant subring by applying Remark \ref{GenInvNoz} with
$z=u^2,x=(uv)^{{m}/{2}},y=(vu)^{{m}/{2}},w=t^2$.
That the proposed generators are necessary follows from the same
bigraded Hilbert series argument as in the $A_{3,\epsilon}^-$ case,
again by Lemma \ref{lem:notAS} this algebra is not AS regular.

For the algebra $A^-_{5,\varepsilon}$, we may conclude as for
the algebra $A^-_{2,\varepsilon}$ since the base of the Ore 
extension as well as the action on it are identical,
$t^2$ is invariant, and the total degree of an invariant must
be even.
\end{proof}

\begin{theorem} 
Let $1 \leq j \leq 5$. The fixed rings for the inner-faithful actions of $H=\cA_{4m}$, $m$ even (see Theorem \ref{A4mEvenIF})  on $A_{j,\epsilon}^+$ are not AS regular:
\begin{align*}
(A^+_{1,\varepsilon})^{H}= &
\kk\langle u^2v^2, u^m+v^m, t^2, uv(u^m-v^m)\rangle,
\\
(A^+_{2,\varepsilon})^{H}= & \kk \langle u^4,(uv)^{m/{2}} - (vu)^{{m}/{2}}, u^2((uv)^{{m}/{2}} + (vu)^{{m}/{2}}), t^2\rangle,
\\
(A^+_{3,\varepsilon})^{H}=& \kk \langle u^2v^2, u^m+v^m, uv(u^m-v^m), uvt^2, (u^m-v^m)t^2,t^4\rangle
\\
(A^+_{4,\varepsilon})^{H}= &
\kk \left\langle
\begin{array}{ll}
u^4,(uv)^{{m}/{2}} - (vu)^{{m}/{2}}, u^2((uv)^{{m}/{2}} + (vu)^{{m}/{2}}), \\
u^2t^2, ((uv)^{{m}/{2}}+ (vu)^{{m}/{2}})t^2,t^4
\end{array}
\right\rangle\\
(A^+_{5,\varepsilon})^{H}=& \kk \langle u^4,(uv)^{{m}/{2}} - (vu)^{{m}/{2}}, u^2((uv)^{{m}/{2}} + (vu)^{{m}/{2}}), t^2 \rangle.
\end{align*}
Hence $\cA_{4m}$, $m$ even,  is not a reflection Hopf algebra for any of the algebras $A^+_{j,\epsilon}$.

\end{theorem}
\begin{proof}To determine the invariants, we proceed as in 
Theorem~\ref{A4mfixedringsEven}, keeping in mind that all invariants
must involve only even powers of $t$ by Proposition 
\ref{prop:noOddDegree}.

The action of $H=\cA_{4m}$ on $A_{1,\epsilon}^+$ is given by 
\begin{align*}
s_+(u^pv^qt^r)&=(-1)^{pq}u^qv^pt^r,\\
s_-(u^pv^qt^r)&=(-1)^{pq}\lambda^{i(p-q)}u^qv^pt^r.
\end{align*}
for $r$ even, and so the invariants are of the form
\[ \sum_{p,q,r\;\mathrm{even}} \alpha_{p,q,r}(u^qv^q (u^{km} + (-1)^{pq} v^{km}))t^r,
\]
where $p,q,r$ are nonnegative integers, with $p = q + km$ for some $k$, and $r$ is even.  Relabeling our coefficients,
one may write the previous display as
\[ \sum_{k,q,r} \alpha_{k,q,r}(-1)^{\binom{q}{2}}(uv)^q ((u^m)^k + (-1)^{q+k} (-v^{m})^k)t^{2r}.
\]
Now we apply Remark \ref{GenInvsEven}, with $z=uv,x=u^m,y=-v^m,w=t$.

The action of $H=\cA_{4m}$ on the even degree monomials of the base 
ring of $A_{2,\epsilon}^+$ is given as follows. When $p$ is an even 
integer and $q$ is arbitrary:
\begin{equation}
\begin{split}
s_+(u^p(vu)^qv) & = (-1)^{{p}/{2}}\lambda^{iq}u^{p+1}(vu)^q \\
s_+(u^{p+1}(vu)^q) & = (-1)^{{p}/{2}}\lambda^{-iq}u^p(vu)^qv \\
s_-(u^p(vu)^qv) & = (-1)^{{p}/{2}}\lambda^{-i(q+1)}u^{p+1}(vu)^q \\
s_-(u^{p+1}(vu)^q) & = (-1)^{{p}/{2}}\lambda^{i(q+1)}u^p(vu)^qv
\end{split}
\end{equation}
Since $t^2$ is invariant and due to the action of $a$,
we need only consider which elements in $k\langle u,v\rangle$ of 
even degree are fixed.  Therefore, let
\[
F=\sum_{q,p\;\text{odd}}\alpha_{p,q}u^p(vu)^qv+\sum_{q,p\;\text{odd}}\beta_{p,q}u^{p-1}(vu)^{q+1}.
\]
be an invariant in only $u$ and $v$ where each $p$ is odd.
The actions of $s_+$ and $s_-$ imply that 
$q+1 \equiv 0 \pmod{m/2}$. Writing $q+1 = \ell m/2$ for some integer $k$, and again relabeling
our coefficients, we have
\[
F = \sum_{k, \ell} \alpha_{k,\ell}u^{2k} \left( ((uv)^{{m}/{2}})^\ell + (-1)^{k + \ell}((vu)^{{m}/{2}})^\ell \right).
\]
We conclude by using Remark \ref{GenInvsEven} with
$z=u^2,x=(uv)^{{m}/{2}},y=(vu)^{{m}/{2}},w=t$.  

For $A_{3,\varepsilon}^+$ one shows that:
\begin{align*}
s_+(u^pv^qt^{2r})&=(-1)^{pq+r}u^qv^pt^{2r}, &
s_-(u^pv^qt^{2r})&=(-1)^{pq+r}\lambda^{i(p-q)}u^qv^pt^{2r}.
\end{align*}
An element $\sum\alpha_{p,q,r}u^pv^qt^{2r}$ is an invariant if and only if $p=q+km$ and the element has the form
\[
\sum_{p,q,r}\alpha_{p,q,r}u^qv^q(u^{km}+(-1)^{pq+r}v^{km})t^{2r},
\]
which can be written as
\[
\sum_{q,k,r}\alpha_{q,k,r}(-1)^{\binom{q}{2}}(uv)^q(u^{km}+(-1)^{q+k+r}(-v^{m})^k)t^{2r}.
\]
Now one concludes using Lemma \ref{GenInv} by setting $z=uv,x=u^m,y=-v^m,w=t^2$.

For $A_{4,\varepsilon}^+$ one can check that an invariant must have the form
\[
\sum_{p,k,r}\alpha_{p,k,r} u^{2p}\left(((uv)^{{m}/{2}})^k+(-1)^{p+r+k}((vu)^{{m}/{2}})^k\right)t^{2r}.
\]
Now we conclude using Lemma \ref{GenInv} with $z=u^2,
x=(uv)^{{m}/{2}},y=(vu)^{{m}/{2}},w=t^2$.

For the algebra $A^+_{5,\varepsilon}$, we may conclude as for
the algebra $A^+_{2,\varepsilon}$ since the base of the Ore 
extension as well as the action on it are identical,
$t^2$ is invariant, and the total degree of an invariant must
be even.

That each of the proposed generating sets for $A_{1,\varepsilon}^+,
A_{2,\varepsilon}^+$ and $A_{5,\varepsilon}^+$ is minimal follows from Hilbert
series arguments as in the proofs of Theorems \ref{thm:B4mPlus}
\ref{A4mfixedringsEven}.  For $A_{3,\varepsilon}^+$ and $A_{4,\varepsilon}^+$,
note that all three generators in $u$ and $v$ alone are necessary, as well
as at least one generator involving a $t$.  Therefore Lemma \ref{lem:notAS}
shows that they are not AS regular.

\end{proof}

%%%%%%%%%%%%%%%%%%%%%%%%%%
%  Sections 7: Extension rings
%%%%%%%%%%%%%%%%%%%%%%%%%%%%%%
\section{Extension rings}
\label{sec:ext}
We conclude this paper with a result which allows one to extend a Hopf action
on an algebra $A$ to an action on an Ore extension of $A$.  This allows one to extend
the Hopf actions in the earlier sections to algebras of larger GK dimension.

\begin{proposition}
Let $H$ be a Hopf algebra that acts linearly on an algebra $A$, and let
$\sigma : A \to A$ be an automorphism of $A$ which is also an $H$-module homomorphism.  Let $\kk t$ be the trivial $H$-module.  Then the action
of $H$ extends to $A[t;\sigma]$, and $A[t;\sigma]^H = A^H[t;\sigma]$.
Further, if the action of $H$ is inner-faithful on $A$ then it is inner-faithful
on $A[t;\sigma]$.
\end{proposition}

\begin{proof}
Recall that the trivial $H$ module $\kk t$ satisfies $ht = \epsilon(h)t$.  To prove
that $H$ acts on $A[t;\sigma]$, we need only check that the set of relations
from the Ore extension are closed under the $H$ action.  To see this, let $h \in H$
and $a \in A$.  Write $\Delta(h) = \sum h_{(1)} \otimes h_{(2)}$. Then one has
\begin{eqnarray*}
h(ta - \sigma(a)t) & = & \sum \left(\epsilon(h_{(1)})th_{(2)}(a) - h_{(1)}\sigma(a)\epsilon(h_{(2)})t\right) \\
                   & = & \sum \epsilon(h_{(1)})th_{(2)}(a) - \sum \sigma(h_{(1)}a)\epsilon(h_{(2)})t \\
                   & = & t\left(\sum \epsilon(h_{(1)})h_{(2)}\right)a - \sigma\left( \left(\sum h_{(1)}\epsilon(h_{(2)})\right)a\right) t \\
                   & = & t(ha) - \sigma(ha)t.
\end{eqnarray*}
To see that $A[t;\sigma]^H = A^H[t;\sigma]$, suppose that $\sum a_i t^i$
is in $A[t;\sigma]^H$.  Then one has
\begin{eqnarray*}
\sum_i \epsilon(h)a_i t^i & = & h\left(\sum_i a_i t^i\right) 
  =  \sum_i \sum h_{(1)}a_ih_{(2)}t^i \\
  & = & \sum_i \sum h_{(1)}a_i\epsilon(h_{(2)})t^i 
   =  \sum_i \left(\sum h_{(1)}\epsilon(h_{(2)})\right)a_it^i \\
  & = & \sum_i ha_it^i \\
\end{eqnarray*}
Since $\{t^i\}$ form a basis of $A[t;\sigma]$ as an $A$-module, the result follows.
The inner-faithful claim is clear.
\end{proof}

\begin{remark}
Note that of course any scalar map $\sigma : A \to A$ is a Hopf module map,
and if $A$ is connected graded $\kk$-algebra with $\kk$ algebraically closed and 
generated in degree one by an irreducible $H$-module, these will be the \emph{only} graded $H$-module endomorphisms of $A$.
\end{remark}

\begin{remark} In order to provide an action of $H$ on the Ore 
extension $A[t;\sigma]$, the hypothesis that $\kk t$ { is} a
trivial module is stronger than what is actually required.  Indeed let 
$\kk t$ be any one-dimensional module and let $\rho : H \to \kk$ be the 
algebra map associated to the representation.
Then the proof above shows that $H$ acts on $A[t;\sigma]$ provided
$\sigma$ is $H$-linear and for all $h \in H$, one has
$\sum \rho(h_{(1)})h_{(2)} = \sum h_{(1)}\rho(h_{(2)})$.
If we let $\iota$ denote the unit of $H$, then this condition is 
equivalent to saying that $(\iota\rho) * \id_H = \id_H *\ (\iota\rho)$, 
where $*$ denotes the convolution product of the algebra $\Hom_k(H,H)$.

This holds if $\kk t$ is the trivial $H$-representation given by the 
counit $\epsilon$ (since $\iota\epsilon$ is the identity
of $\Hom_k(H,H)$), or if $H$ is a commutative or cocommutative Hopf 
algebra.  One may also check that many of the one-dimensional 
representations of the Hopf algebras considered in this paper also 
satisfy this condition.  However, one may check that the representation 
$T_{1,1,-1}$ of $\cA_{4m}$ ($m$ even) does not.

Lastly we note that of course the claim regarding the fixed ring of 
these more general actions no longer holds, and there can often be many 
more elements that are fixed by the action of $H$ than just 
$A^H[t^m;\sigma^m]$ (where $m$ is the order of the representation in
the Grothendieck ring of $H$).  Determining the invariant subrings of 
such actions would be of significant interest.
\end{remark}
\noindent
 {\bf Acknowledgement:} The authors thank the referee for several helpful suggestions.
\bibliography{biblio}
\bibliographystyle{amsplain}

\end{document}